\documentclass[letterpaper]{article}
\usepackage{graphicx} 

\usepackage{amsmath}
\usepackage{amssymb}
\usepackage{amsthm}
\usepackage{stmaryrd}
\usepackage{hyperref}
\hypersetup{final}                                                       \hypersetup{
    colorlinks,
    linkcolor={red!50!black},
    citecolor={green!50!black},
    urlcolor={blue!80!black}
}
\usepackage{empheq}
\usepackage{enumitem}
\usepackage{color}
\usepackage{version}
\usepackage{amsfonts}
\usepackage{dsfont}
\usepackage{bbm}
\usepackage{a4wide}
\usepackage{overpic}
\usepackage{blindtext}
\usepackage{titlesec}
\usepackage{xcolor}
\usepackage[normalem]{ulem}
\usepackage{authblk}
\usepackage{mathtools}

\title{Can one condition a killed random walk to survive?}

\author[1]{Lucas Rey}
\affil[1]{CEREMADE, CNRS, UMR 7534, Université Paris-Dauphine, PSL University, 75016 Paris, France, \textit{lucas.rey@dauphine.psl.eu}}
\author[2]{Augusto Teixeira}
\affil[2]{IMPA, Estrada Dona Castorina, 110 - Rio de Janeiro, Brazil and IST, Av. Rovisco Pais, Nº 1,
Lisbon, Portugal \textit{augusto@impa.br}}

\theoremstyle{plain}
\newtheorem{Thm}{Theorem}[section]

\newtheorem{Prop}[Thm]{Proposition}
\newtheorem{Lem}[Thm]{Lemma}
\theoremstyle{definition}
\newtheorem{Def}[Thm]{Definition}

\theoremstyle{remark}
\newtheorem{Rem}[Thm]{Remark}
\newtheorem{Ex}[Thm]{Example}

\newtheorem{Qu}[Thm]{Question}

\newcommand{\NN}{\mathbb{N}}
\newcommand{\R}{\mathbb{R}}
\newcommand{\Z}{\mathbb{Z}}
\newcommand{\N}{\mathbb{N}}
\newcommand{\T}{\mathbb{T}}
\newcommand{\Proba}{\mathbb{P}}
\newcommand{\E}{\mathbb{E}}
\newcommand{\Q}{\mathbb{P}}
\newcommand{\K}{\mathbf{k}}
\newcommand{\h}{\mathbf{h}}

\newcommand{\F}{\mathcal{F}}
\newcommand{\G}{\mathcal{G}}
\newcommand{\s}{\mathbf{s}}
\newcommand{\B}{\mathbf{b}}

\newcommand{\eps}{\varepsilon}
\newcommand{\EK}{E_{\K}}
\newcommand{\RR}{\mathcal{R}}

\def\arraypar#1{\parbox[c]{\textwidth - 2cm}{\centering #1}}

\RequirePackage{environ}
\NewEnviron{display}{
  \begin{equation}
    \begin{array}{c}
      \arraypar{\BODY}
    \end{array}
  \end{equation}
}

\newcommand{\be}{\begin{equation}}
\newcommand{\ee}{\end{equation}}
\newcommand{\fix}[1]{{\textcolor{black}{#1}}}

\begin{document}
\mathtoolsset{showonlyrefs}
\maketitle
\begin{abstract}
    We consider the simple random walk on $\Z^d$ killed with probability $p(|x|)$ at site $x$ for a function $p$ decaying at infinity. Due to recurrence in dimension $d=2$, the killed random walk (KRW) dies almost surely if $p$ is positive, while in dimension $d \geq 3$ it is known that the KRW dies almost surely if and only if $\int_0^{\infty}rp(r)dr = \infty$, under mild technical assumptions on~$p$.
    In this paper we consider, for any $d \geq 2$, functions $p$ for which the random walk will die almost surely and we
    ask ourselves if the KRW conditioned to survive is well-defined.
    More precisely, given an
    exhaustion $(\Lambda_R)_{R \in \N}$ of $\Z^d$, does the KRW conditioned to leave $\Lambda_R$ before dying converges in distribution towards a limit which does not depend on the exhaustion? We first prove that this conditioning is well-defined for $p(r) = o(r^{-2})$, and that it is not for $p(r) = \min(1, r^{-\alpha})$ for $\alpha \in (14/9,2)$. This question is connected to branching random walks and the infinite snake.
    More precisely, in dimension $d=4$, the infinite snake is related to the KRW with $p(r) \asymp (r^2\log(r))^{-1}$, therefore our results imply that the infinite snake conditioned to avoid the origin in four dimensions is well-defined.

    \bigskip

\end{abstract}



\section{Introduction}
The simple random walk (SRW) on $\Z^d$ is known to be transient in dimension $d \geq 3$, recurrent in dimensions $1$ and $2$ (Polya's theorem). In dimension $d \geq 3$, the SRW conditioned to avoid the origin is clearly a well-defined object. But perhaps more surprisingly, this conditioning also makes sense at the critical dimension $d=2$ \fix{and below ($d=1$)}, that means: the SRW conditioned to avoid the origin in $\Z^\fix{d}$ can be defined as the Doob-transform of the SRW by the potential kernel $a$. It can also be defined as a weak limit.
\begin{Def}\label{def:exhaustion}
    An increasing sequence $\Lambda = (\Lambda_R)_{R \in \N}$ of finite subsets of $\Z^d$ is called an \emph{exhaustion} of $\Z^d$ if for all $x \in \Z^d$, there exists $R$ such that $x \in \Lambda_R$. 
\end{Def}
Perhaps the simplest example of an exhaustion is given by $(B(R))_{R \in \N}$, the intersections with $\Z^d$ of the \fix{E}uclidean balls of radius $R$. 
This was precisely the exhaustion used in \cite{Gantert} in order to define the SRW on $d = 2$ conditioned to avoid the origin as the limit in distribution of the SRW conditioned to escape $\Lambda_R$ before hitting the origin.
This conditioned random walk proved useful to define the random interlacement in two dimensions in \cite{2dRI}, and it is also interesting in its own right, see for example \cite{ConditionedWalk}, \cite{Gantert}.
One of the consequences of our Theorem \ref{thm.cv} below is to establish the independence of the limit with respect to the choice of $\Lambda_R$.

An important motivation for our work is to find a similar construction for the infinite snake of \cite{Curien}. Given $\mu$ a probability measure on $\NN$ with $\E[\mu]=1$, denote by $T_{\infty}$ the critical Galton-Watson tree with offspring distribution $\mu$ conditioned to survive. The \emph{infinite snake} is the simple random walk on $\Z^d$ indexed by $T_{\infty}$ (see Section \ref{sec:snake} for more details). It is shown in \cite{Curien} that this infinite snake is recurrent for $d\leq 4$ and transient \fix{for} $d\geq 5$, which is reminiscent of Polya's theorem for the SRW. Moreover, the hitting probabilities for the \fix{infinite} snake in dimension $d=4$ show a critical behavior, see \cite{Zhu}. This inspires our first main result, cf Theorem \ref{thm:cv:snake}:

\begin{display}\label{thm.snake}
    When $d= 4$ and $\mu$ has finite variance, given any exhaustion $\Lambda$, the infinite snake conditioned to avoid the origin before escaping $\Lambda_R$ converges in distribution when $R \to \infty$ towards a limit which is independent of the exhaustion. 
\end{display}
\begin{Rem}
The expression ``before escaping $\Lambda_R$'' in the above statement is somewhat unprecise, since we have not defined an order on $T_{\infty}$: this will be carefully done in Section \ref{sec:snake}.
\end{Rem}
The way we prove this result is by first phrasing it in terms of a killed random walk (KRW) with decaying killing rate, see \cite{Zhu} and our Section \ref{sec:snake} for the connection between these two perspectives.
Moreover, the KRW is an interesting object of its own and we introduce it in the definition below. Let $S$ denote the simple random walk (SRW) on $\Z^d$.
\begin{Def}\label{def:krw}
    Let $\K: \Z^d \to [0,1]$ a \emph{killing function}. The killed random walk (KRW) $S^{\K}$ is the SRW $S$ killed at each step with probability $\K$. More precisely, if $\Delta$ is an additional state called the cemetery, the KRW is the Markov chain on $\Z^d \cup \{ \Delta \}$ with transition probabilities: $Q^{\K}(\Delta, \Delta) = 1$, $Q^{\K}(x, \Delta) = \K(x)$ if $x \in \Z^d$ and otherwise
    \begin{equation}
        Q^{\K}(x,y) = \frac{1-\K(x)}{2d}\mathds{1}_{x \sim y}.
    \end{equation}
\end{Def}
For $A \subset \Z^d \fix{\cup \{\Delta\}}$, define the stopping time $\tau(A) = \inf \{n \geq 0, S_n^{\fix{\K}} \in A\}$. In particular, $\tau(\Delta)$ denotes the killing time of the KRW. We denote by $E_{\K} \subset \Z^d$ the sites $x \in \Z^d$ from which there exists a path \fix{$x = x_0, x_1, x_2, \dots$} along which \fix{$\K(x_i) < 1$ for all $i \geq 0$ and $|x_i| \overset{i \to \infty}{\longrightarrow}\infty$}, and we always assume that $E_{\K}$ is connected, see Section \ref{sec:preliminaries} for more details. When $k(x) < 1$ for every $x \in \mathbb{Z}^d$, then $E_{\K} = \Z^d$: the reader can keep this case in mind.
\fix{Following~\cite{PemVol98}, we say that the KRW is \emph{trapped} if for all $x \in E_\K$, $\Proba_x[\tau(\Delta) < \infty]  = 1$.
Otherwise, the KRW is \emph{not trapped}: in this case, since $\K < 1$ on $E_{\K}$ (and because it is a connected set), $\Proba_x[\tau(\Delta) < \infty]  < 1$ for all $x \in E_\K$.}
An important question at this point is whether $S^{\K}$ \fix{is trapped or not}. In dimension $d \leq 2$, the answer is easy: if the killing function is not identically $0$ (that is $\K(x) >0$ for some $x \in \Z^2$), since the SRW is recurrent it will almost surely visit $x$ infinitely many times, and therefore $S^{\K}$ dies with probability one. In dimension $d \geq 3$, the question is subtle. 
\fix{A first partial answer was given in \cite{Hollander}, which was completed and extended in~\cite{PemVol98}.
For $d \geq 3$, $x,y \in \Z^d$, let $g(x,y)$ denote the Green function of the SRW, that is
\be\label{eq:def:green}
    g(x,y) = \sum_{n = 0}^{\infty} \Proba_x[S_n = y].
\ee
It is well-known that $g(x,y) \sim a_d |x-y|^{2-d}$ as $|x-y|\to \infty$ for some (explicit) constant $a_d$ (see for example Theorem 1.5.4 of \cite{Lawler}).
\begin{Thm}[Pemantle, Volkov]\label{th.rec}
    Assume that $d \geq 3$ and that $\K$ is \emph{approximately invariant under rotation}, that is there exist $x_0 \in E_\K$ and $C, C' > 1$ such that
    \be\label{eq:approx:rot:inv}
        \frac{1}{C}g(x_0,y) \leq g(x_0,x) \leq Cg(x_0,y) \implies \frac{1}{C'}\K(x) \leq \K(y) \leq C'\K(x).
    \ee
    Then the KRW is trapped if and only if~$\sum_{x \in \Z^d} g(x_0,x)\K(x) = \infty$, or equivalently $\sum_{x \in \Z^d} |x|^{2-d}\K(x) = \infty$.
\end{Thm}
The theorem of~\cite{PemVol98} holds for more general Markov chains under a \emph{reasonable annuli condition}, which is satisfied by the SRW on $\Z^d$, see their section 3 for details.}
Note that by using monotonicity arguments, this theorem allows us to study much more general
killing functions than only rotationally invariant ones.

\bigskip

A \fix{non-trapped} KRW conditioned to survive is a perfectly well-defined object because the event $\fix{\big\{}\tau(\Delta) = \infty\fix{\big\}}$ has positive probability. The goal of this work is to extend this definition to the \fix{trapped} case, but this cannot be done without care, since the event $\fix{\big\{}\tau(\Delta) = \infty\fix{\big\}}$ has probability zero.

\begin{Rem}\label{rem:1d}
    It follows from a simple Gambler's ruin computation that the \fix{KRW} $S^{\K}$ on $\mathbb{Z}$ associated with the killing function $\K(x) = (1/2)\mathbbm{1}_0 (x)$ conditioned to leave the segment $[-b_-R, b_+R]$ before dying satisfies
    \be
        \Proba_0\big[S^{\K}_1 = 1|\min\big(\tau(b_+R),\tau(-b_-R)\big) < \tau(\Delta)\big] = \frac{1}{4}+\frac{b_-}{2(b_++b_-)}.
    \ee
    This shows that the limit in distribution of the KRW conditioned to leave $[b_-R, b_+R]$ depends on the ratio $b_+/b_-$, and hence it is not independent of the way in which we exhaust the underlying space. This remark motivates the following definition.
\end{Rem}

\begin{Def}\label{def:conditioned:KRW}
    We say that the KRW conditioned to survive is well-defined if for any exhaustion $\Lambda$, for any $x \in \EK$, the KRW started at $x$ conditioned to hit $(\Lambda_R)^c := \Z^d \setminus \Lambda_R$ before dying converges in distribution when $R \to \infty$ towards a Markov chain on $\EK$ which is independent of the exhaustion.
\end{Def}
In the above definition, the convergence in distribution is understood as the convergence of the first steps of the walk  $(S_0, S_1, \dots, S_n)$ for $n$ fixed, as $R$ goes to infinity.
This definition generalizes the common definition in the \fix{non-trapped} case since for any exhaustion the event $\fix{\big\{}\tau(\Delta) = \infty\fix{\big\}} = \bigcap_{R \in \N} \fix{\big\{}\tau(\Lambda_R^c) < \tau(\Delta)\fix{\big\}}$ has positive probability.
\begin{Rem}
    Definition \ref{def:conditioned:KRW} is one among many natural ways to define this object. We could, for example, condition the walk to survive until time $n$ and let $n \to \infty$.
    The equivalence of these two approaches in our more general setting is not studied in this article.
    These two approaches were shown to be equivalent in \cite{Doney} for random walks on $\N$ with zero drift, when conditioned to avoid the origin. Note that in this case, Remark \ref{rem:1d} does not apply since $\mathbb{N}$ has a single end, therefore \fix{all} exhaustion\fix{s} \fix{are} equivalent.
\end{Rem} 
It has been known for a long time that this sort of convergence is equivalent to a ratio limit theorem (see for example \cite{Doney}). Before stating this characterization, we need a definition.
\begin{Def}
    A function $a : \Z^d \to \R$ is massive harmonic if for all $x \in \Z^d$,
    \be
        a(x) = \frac{1-\K(x)}{2d}\sum_{y \sim x}a(y)
    \ee
\end{Def}
\begin{Prop}\label{prop:ratio:limit}
    The KRW conditioned to survive is well-defined (in the sense of Definition \ref{def:conditioned:KRW}) if and only if 
    there exists a massive harmonic function $a: \Z^d \to \R$, positive on $\EK$, such that for every exhaustion $\Lambda$, for all $x,y \in \EK$,
    \be\label{eq:cv:ratio}
        \frac{\Proba_y[\tau(\Lambda_R^c) < \tau(\Delta)]}{\Proba_x[\tau(\Lambda_R^c) < \tau(\Delta)]} \overset{R \to \infty}{\longrightarrow} \frac{a(y)}{a(x)}.
    \ee
    In this case, the KRW conditioned to survive is the Doob transform of $S^{\K}$ by $a$.
\end{Prop}
We will provide a proof of this classical result in Section \ref{sec:preliminaries}.  

\begin{Ex}\label{ex:krw:2d}
    If $d=2$ and $\K = \mathbbm{1}_0 (x)$, the KRW corresponds to the SRW in $\Z^2$ killed at the origin. In this case, $\EK = \Z^2 \setminus \{0\}$ is connected and Proposition \ref{prop:ratio:limit} is known to hold for the specific exhaustion $\Lambda_R = B(R)$ (see for example Lemma 3.3(i) of \cite{2dRI}) where $a$ is the potential kernel in two dimensions.
\end{Ex}
Our first main result on KRW\fix{s} gives the existence of the KRW conditioned to survive when the killing function decays fast enough.
\begin{Thm}\label{thm.cv}
    Let $d \geq 2$ and assume that $\K: \Z^d \to [0, 1]$ satisfies $\K(x) = o(|x|^{-2})$ when $|x| \to \infty$. Let $\Lambda$ be any exhaustion of $\Z^d$. Then if we fix any $x_0$ in $\EK$, the limit 
    \begin{equation}\label{prop.cv}
        a(x) = \lim_{R \to \infty} \frac{\Proba_x[\tau(\Lambda_R^c) < \tau(\Delta)]}{\Proba_{x_0}[\tau(\Lambda_R^c) < \tau(\Delta)]}
    \end{equation}
    exists for all $x \in \Z^d$ and does not depend on the sequence $\Lambda$. It is positive on $\EK$, zero on $\EK^{c}$ and massive harmonic on $\EK$.  Hence, by the characterization of Proposition \ref{prop:ratio:limit}, the KRW conditioned to survive is well-defined.
\end{Thm}

This theorem will be proved in Section \ref{proof3d} for the case $d \geq 3$ and in Appendix \ref{proof2d} for the case $d=2$. We chose to separate the proofs because the proof for the case $d=2$ is more technical though it involves no new ideas.
\begin{Rem}
    This theorem covers the case of the SRW conditioned to avoid the origin in dimension $d=2$ and extends Example \ref{ex:krw:2d} to any exhaustion: it corresponds to the choice $d=2$ and $\K(x) = \mathbbm{1}_0 (x)$ which indeed satisfies $\K(x) = o(|x|^{-2})$. In this case, the massive harmonic function provided by the theorem is the potential kernel (see the reference in Example \ref{ex:krw:2d}).
\end{Rem}
\begin{Rem}
    When the KRW is \fix{not trapped}, the statement of the theorem  becomes trivial. In dimension $2$, this theorem is not \fix{trivial} since the KRW is always \fix{trapped}. In dimension $d \geq 3$, this theorem is also not \fix{trivial in many cases} since there exist functions $\K(x) = p(|x|)$ with $\sum_{x \in \Z^d} |x|^{2-d}p(|x|) = \infty$ or equivalently $\int_{r = 1}^{\infty}rp(r)dr = \infty$ (and therefore, according to Theorem \ref{th.rec}, the KRW is \fix{trapped}) while $p(|x|) = o(|x|^{-2})$. In particular, the infinite snake is related to one such case: namely $d=4$ and $\K(x) \asymp 1/(|x|^2\log(|x|)$.
\end{Rem}

\begin{Rem}
    It would be interesting to investigate in more detail \fix{the conditions under which such conditioning is robust} with respect to the limiting procedure.
    When the killing function has compact support, we have seen that the robustness holds in $d \geq 2$, but not $d = 1$.
    It is tempting to attribute this fact to the criticality of $d = 2$, for which the random walk is ``barely recurrent''.
    This raises the question of whether this universality holds for more general natural graphs with dimension two and whether it fails as soon as the dimension is strictly smaller than two.   
\end{Rem}

In the next theorem, we give an indication that this robustness of the limit holds only at criticality, not with respect to the dimension, but rather when regarding the decay of the killing rate.

\begin{Thm}\label{krwcounterex}
    Let $d = 2$, $\alpha \in ]14/9, 2[$ and $\K(x) = \min(1, |x|^{-\alpha})$ for $x \in \Z^2$. Then, the KRW conditioned to survive is not well-defined in the sense of Definition \ref{def:conditioned:KRW}. More precisely, there exist two exhaustions of $\Z^2$, $\Lambda^1$ and $\Lambda^2$ and $x, y \in \Z^2$ such that
    \begin{equation}
        \lim_{R \to \infty} \frac{\Proba_x\big[\tau\big((\Lambda^1_R)^c\big) < \tau(\Delta)\big]}{\Proba_y\big[\tau\big((\Lambda^1_R)^c\big) < \tau(\Delta)\big]} \neq \lim_{R \to \infty} \frac{\Proba_x\big[\tau\big((\Lambda^2_R)^c\big) < \tau(\Delta)\big]}{\Proba_y\big[\tau\big((\Lambda^2_R)^c\big) < \tau(\Delta)\big]}
    \end{equation}
    (either one of these limits does not exist, or if they both exist they are different). 
\end{Thm}
This theorem is proved in Section \ref{appendixA}.
\begin{Rem}
    We did not try to optimize the lower bound for $\alpha$, and we believe that this holds for $\alpha \in [0,2[$. We state the result only in dimension $d=2$ because the proof uses the skew-product decomposition which takes a particularly nice and explicit form in dimension $d=2$. We believe that it can be extended to all dimensions $d \geq 2$ with the same techniques.
\end{Rem}
\begin{Rem}
    We will see in the proof that we can take for example $\Lambda^1_R = B(R) \cup \big(\Z_- \times \Z\big)$, $\Lambda^2_R =  B(R) \cup \big(\Z_+ \times \Z\big)$ for all $R > 0$ and $x = -y = (r,0)$ for $r$ large enough but fixed. These sets are not finite, but a counter-example with finite subsets can easily be built from them, see Remark \ref{rem:exhaustion:finite}.
\end{Rem}
\begin{Rem}
    Along the specific exhaustion $\Lambda_R = B(R)$, a limit in distribution might exist due to the rotational invariance (it exists at least for the killed \fix{B}rownian motion, see Section \ref{counterex} for a definition). This theorem shows that the limit is not universal, since it depends on whether we “shift the infinity to the right” or “to the left”.
\end{Rem}

\paragraph{Open questions.} Putting together Theorem \ref{thm.cv} and Theorem \ref{krwcounterex} we know that for $\K(x) = o(|x|^{-2})$ the KRW conditioned to survive is well-defined, while it \fix{is} not for $\K(x) = \min(1,|x|^{-\alpha})$ with $\alpha < 2$: therefore the exponent $2$ is sharp. However, it would be interesting to know what happens around $|x|^{-2}$, in particular because \cite{brw3} shows that the killing function associated to the infinite snake in dimension $d \leq 3$ is $\K(x) \asymp |x|^{-2}$, so our theorems neither prove nor disprove the existence of the infinite snake conditioned to avoid the origin.

With our definition, the existence of the KRW conditioned to survive is related \fix{to} the existence of massive harmonic functions. Is our definition equivalent to the uniqueness of massive harmonic functions, with some decay assumptions? For random walks killed with constant rate, there is a notion of $t$-Martin boundary, generalizing the usual Martin boundary, see \cite{Woess} Section 24 and \cite{MartinBoundary} for a more recent work on the subject. Can our definition of KRW conditioned to survive be rephrased in terms of a generalized Martin boundary of the KRW?

\paragraph{Previous works}
This article is at the intersection of two historic lines of research: research in statistical physics on random walks in a field of traps and research on random walks conditioned to avoid some fixed finite set in dimension $1$ and $2$. The random walk in a field of trap\fix{s} (that we call here KRW) is an object coming from statistical physics. A review of the physical models and their mathematical study can be found in \cite{trapping}. A more recent review with less physical background is \cite{mobile/immobile}. The continuous analog of this model is the Brownian motion in a field of trap\fix{s} which is also a well studied model, see for example \cite{Sznitman}. The literature focuses mainly on the quenched and annealed regimes. In the quenched case, a trap is put in every $x \in \Z^d$ with probability $\K(x) \in [0,1]$, and when the KRW hits a trap it dies. In the annealed case, every time the KRW is at site $x$, it dies with probability $\K(x)$. This corresponds to resampling the distribution of the traps at each step. Here we will focus on the annealed case with killing probability decaying at infinity while most of the literature focuses on the case of traps that are homogeneously distributed among the space. 
Our model was first studied in~\cite{Hollander} which develops \fix{necessary and sufficient conditions for trapping/non-trapping on $\Z^d$}, both in the quenched and annealed cases.
\fix{Their results were extended in~\cite{PemVol98}, where the authors work with more general Markov chains, prove equivalence of the quenched and annealed problems and lighten the regularity assumptions of~\cite{Hollander}.}
An article with a similar flavor to our article is \cite{Madras} which also studies the annealed case with a killing function decaying at infinity. In the \fix{trapped} case (corresponding to $d=2$) they give criteria for the KRW \fix{to} be long-lived (i.e. the KRW coming from infinity has a positive chance of visiting a neighborhood of the origin before being killed). 
The most simple case of the KRW is the walk killed with probability $1$ when it hits a fixed finite set $A$ (corresponding to the killing function $\K(x) = \mathbbm{1}_A(x)$). 
This model is the same in the annealed or quenched case, and the \fix{trapping/non-trapping} of the KRW is equivalent to the recurrence/transience of the set $A$ for the SRW. In this simple case, determining the surviving probability boils down to computing the probability of hitting the set $A$, for which all the tools of capacity theory are available. A historic\fix{al} reference, either in the recurrent or transient case, is the book \cite{Spitzer} (Chapters II, IV and VI for the cases $d=2$, $d=1$, $d\geq 3$).\par
This leads us to the other line of research to which our article belongs, in the most simple setting of a single trap placed at the origin, but for general random walks. 
It has been known for a long time that the existence of a random walk conditioned to survive is linked to the existence of harmonic functions, and that when the random walk conditioned to survive exists it is the h-transform of the random walk by a massive harmonic function.
In one dimension, the problem of finding harmonic functions on the half-line of $\Z$ is settled in Theorem E3 of \cite{Spitzer} for recurrent aperiodic random walks. 
A broad generalization, still in the one-dimensional case, is \cite{Doney} which studies random walks (centered or not) on the continuous line $\R$ conditioned to stay positive. 
Another interesting paper in the same spirit is \cite{ratio}, which gives ratio limit theorems for periodic irreducible random walks on $\Z^d$, for the probability of hitting finite sets. 
This is exactly the type of convergence results that we will need to prove, and their results are equivalent to showing the existence of the random walk conditioned never to hit a finite set $\Omega$, for two limiting procedure: conditioning on $\fix{\{}\tau(\Omega) = n\fix{\}}$ or $\fix{\{}\tau(\Omega) > n\fix{\}}$ and letting $n \to \infty$. 
We finally mention more recent developments in this area, which do not focus on the convergence but rather on properties of the limiting object, for example \cite{Gantert} which studies the range of the SRW conditioned to avoid the origin in two dimensions.

\paragraph{Proof ideas}
The general idea of the proof of Theorem \ref{thm.cv} in dimension $d\geq 3$ is based on the following picture: even though the killing in the first few steps of the walk is very depend\fix{a}nt of the initial position, and at long time scale the walk dies almost surely, since $\K(x) = o\Big(\frac{1}{|x|^2}\Big)$ we can find an intermediate time scale during which the killing is negligible but the KRW has enough time to forget its initial position. This is the main technical step of the proof (Lemma \ref{lem:zhu}), and it is an adaptation of Lemma 3.1 of \cite{Zhu} to our more general setting. The proof of this lemma is done by generalizing the arguments of \cite{Zhu}. Once this key result is established, the probability to survive on the very long term is the product of two terms: a multiplicative constant depending only on the initial position, and a second factor (independent of the initial position) depending only on how long the walk runs. We combine this decomposition with a Cauchy sequence argument to show the convergence in~\eqref{prop.cv}. In dimension $d=2$, the proof is more technical because the Green function is more delicate to define so we need several intermediate time scales, but no new idea is needed so we \fix{postpone} the proof to Appendix \ref{proof2d}.\par
The proof of Theorem \ref{krwcounterex} is a completely separate matter. We will first prove a \fix{continuous} analog of this statement involving the killed \fix{B}rownian motion (KBM). In $d$ dimensions, the KBM associated to a killing function $\K: \R^d \to \fix{[0, \infty)}$ is the simple Brownian motion in $\R^d$ which from a given position $x$ goes to an additional cemetery state $\Delta$ with rate $\K$: it is the process with generator $A$ such that for all $f \in \mathcal{C}_0^2(\R^d)$,
\be\label{def:generator:KBM}
    Af = \frac{1}{2}\Delta f - \K f
\ee
All the definitions we made for the KRW have an equivalent for the KBM: we postpone the formalities to Section \ref{counterex}. In Section \ref{counterex}, we will prove the following \fix{continuous} analog of Theorem~\ref{krwcounterex}. 
\begin{Thm}\label{thm.counterex}
    Let $d = 2$ and $\K(x) = 1 \wedge \frac{1}{|x|^{\alpha}}$ for $x \in \R^d$ and for some fixed $\alpha \in ]0,2[$. Then, the KBM conditioned to survive does not make sense. More precisely, there exist two exhaustions of $\R^2$, $\Lambda^1$ and $\Lambda^2$ and $x, y \in \R^d$ such that
    \begin{equation}
        \lim_{R \to \infty} \frac{\Proba_x[\tau((\Lambda^1_R)^c) < \tau(\Delta)]}{\Proba_y[\tau((\Lambda^1_R)^c) < \tau(\Delta)]} \neq \lim_{R \to \infty} \frac{\Proba_x[\tau((\Lambda^2_R)^c) < \tau(\Delta)]}{\Proba_y[\tau((\Lambda^2_R)^c) < \tau(\Delta)]}.
    \end{equation}
    (either one of these limits does not exist, or if they both exist they are different).
\end{Thm}
\begin{Rem}
    We will see in the proof that we can take for example for all $\Lambda_R^1 = B(R) \cup (\R_{\fix{-}} \times \R)$, $\Lambda_R^2 = B(R) \cup (\R_{\fix{+}} \times \R)$ for all $R>0$ and $x=-y=(r,0)$ for $r$ large enough but fixed. 
\end{Rem}
The proof relies heavily on the rotational invariance of the KBM and the killing function by using the skew product decomposition (see Section 7.15 of \cite{ItoMcKean} for a good reference). When $\alpha$ is close enough to $2$, we can use Skorokhod's embedding and a strong approximation of the SRW by the Brownian motion to transfer the results from the KBM to the KRW: this is how we will show in Section \ref{appendixA} that Theorem \ref{thm.counterex} implies Theorem \ref{krwcounterex}.

\paragraph{Organization of the paper}
Section \ref{sec:snake} is devoted to the connection between the infinite snake and the KRW. In this section we give a precise meaning to the informal statement \ref{thm.snake} and explain why it is a corollary of Theorem \ref{thm.cv}. Section \ref{proof3d} and Section \ref{counterex} are dedicated to the proof of respectively Theorem \ref{thm.cv} in the case $d \geq 3$ and Theorem \ref{thm.counterex}.  Section \ref{appendixA} proves Theorem \ref{krwcounterex} and Appendix \ref{proof2d} proves the case $d=2$ of Theorem \ref{thm.cv}. The proof in the appendix only add technicalities but no new ideas to the proof contained in the main text.

\paragraph{Acknowledgments}
L. R. was partially supported by the DIMERS project ANR-18-CE40-0033 funded by the French National Research Agency. L. R. also thanks IMPA for its hospitality during two research stays in 2020 and 2022. A.T. was supported by grants ``Projeto Universal'' (406250/2016-2) and ``Produtividade em Pesquisa'' (304437/2018-2) from CNPq and
``Jovem Cientista do Nosso Estado'', (202.716/2018) from FAPERJ.
\fix{The authors thank Serguei Popov for mentioning the article~\cite{PemVol98} after a first version of our paper was posted on arxiv, and an anonymous reviewer for mentioning the same paper and for his/her comments that helped improve the presentation and readability of the paper.}


\section{Preliminaries}\label{sec:preliminaries}

\paragraph{Notation.}
In all this article, \fix{$d \geq 2$} denotes a dimension, $|\cdot|$ denotes the \fix{E}uclidean norm in $\R^d$, $B(r)$ denotes the \fix{E}uclidean ball of radius $r$ centered at $0$.\\
For $f,g: \Z^d \to \R_{\fix{+}}$, we write $f \asymp g$ if there are two constants $c,C > 0$ such that for all $x \in \Z^d$, $cg(x) \leq f(x) \leq Cg(x)$. We will always indicate explicitly on which parameters of the problem the constants $c,C$ depend.\\ For $x, y \in \Z^d$, we write $x \sim y$ if $|x-y|=1$. For $A \subset \Z^d$, we define the complement $A^c := \Z^d \setminus A$ and the outer boundary 
\begin{equation}
    \partial A = \{y \in A^{c} \text{ such that } \exists x \in A \text{ with } x \sim y\}.
\end{equation}
The \emph{simple random walk} (SRW) on $\Z^d$ is the Markov chain on $\Z^d$ with transition probabilities
\begin{equation}
    Q(x,y) = \frac{1}{2d}\mathds{1}_{x\sim y}.
\end{equation}
For $A \subset \Z^d$, we define the stopping times $\tau(A) = \inf \{n \geq 0, S_n \in A\}$ and $\tau^+(A) = \inf \{n > 0, S_n \in A\}$. We will write $\tau(r) = \tau(\partial B(r))$.

Let $\K: \Z^d \to [0,1]$ be a killing function. Recall the definition of the \emph{killed random walk} (KRW) from Definition \ref{def:krw}, \fix{and recall the definition of $E_\K$ following Definition \ref{def:krw}.} Starting from any $x \in E_{\K}^c$, $\Proba_x[\tau(\Delta) = \infty] = 0$. Starting from $x \in E_{\K}$, the KRW can be \fix{trapped or not} depending on the dimension and the asymptotic behavior of $\K$ (see Theorem \ref{th.rec}). \fix{Recall that in this article}, we always assume that $E_{\K}$ is connected to avoid degenerated cases: this is anyway the case for all the killing functions decaying at infinity that we study.

\paragraph{Proof of Proposition \ref{prop:ratio:limit}.}
This proof is not new, we include it to make the reader familiar with the notation. 
\begin{proof}
    If \eqref{eq:cv:ratio} holds, by the simple Markov Property, for all $n \in \N$, \fix{for all} $s_0, ..., s_n \in \EK$, for all $R >0$, 
    \begin{equation}\label{cvRW}
        \begin{aligned}
            \Proba[S^{\K}_0 = s_0, ..., S^{\K}_n = s_n~|~\tau(\Lambda_R^c) < \tau(\Delta)] 
            &\overset{R \to \infty}{\longrightarrow}\frac{a(s_n)}{a(s_0)} \Proba[S^{\K}_0 = s_0, ..., S^{\K}_n = s_n]\\
            &~~~~~~= \prod_{i=0}^{n-1}(1-\K(s_i)) \frac{a(s_n)}{a(s_0)} \Proba[S_0 = s_0, ..., S_n = s_n]
        \end{aligned}
    \end{equation}
    so the KRW conditioned to escape $\Lambda_R$ converges in distribution and hence the KRW is well-defined in the sense of Definition \ref{def:conditioned:KRW}. Reciprocally, if the KRW is well-defined, for all $x \sim y \in \EK$,
    \be
        \Proba[S_0^{\K} = x, S_1^{\K} = y | \tau(\Lambda_R^c) < \tau(\Delta)] = \Proba[S_0^{\K} = x, S_1^{\K} = y] \frac{\Proba_y[\tau(\Lambda_R^c) < \tau(\Delta)]}{\Proba_x[\tau(\Lambda_R^c) < \tau(\Delta)]}
    \ee
    has a limit in $[0, \infty)$ which is independent of $\Lambda$, and since $\Proba[S_0^{\K} = x, S_1^{\K} = y] > 0$ the ratio must converge. By the same argument with $x$ and $y$ interchanged, the inverse ratio must also converge towards a limit in $[0,\infty)$ so both limits are in $(0,\infty)$. If we choose any $x_0 \in \EK$, since $\EK$ is connected, this implies that there exists a positive function $a$ such that for all exhaustion
    \be
        \frac{\Proba_x[\tau(\Lambda_R^c) < \tau(\Delta)]}{\Proba_{x_0}[\tau(\Lambda_R^c) < \tau(\Delta)]} \overset{R \to \infty}{\longrightarrow} \frac{a(x)}{a(x_0)}. 
    \ee
    The massive harmonicity of $x \to \Proba_x[\tau(\Lambda_R^c) < \tau(\Delta)]$ for all fixed $R$ implies the massive harmonicity of $a$, and the equality stated in the Proposition follows by the simple Markov Property.
\end{proof}


\section{Application to the infinite snake}\label{sec:snake}
We start with a few definitions and notations.
\begin{Def}[Trees]
  If $t = (E(t), V(t))$ is a directed graph with vertex set $V(t)$ and edge set $E(t)$, two vertices $u, v \in V(t)$ are connected if $(uv) \in E(t)$. If $u \in V(t)$, the vertices $y \in V(t)$ such that $(uv) \in E(t)$ are called the offspring of $u$. A path is a sequence of connected vertices. A rooted tree $(t, r)$ is a directed graph $t$ with a marked vertex $r(t)=r$ called the root such that for each $u \in V(t)$ there is a unique path from the root to $u$ (we will drop the root from the notation when there is no risk of confusion). Let $\T$ be the class of finite rooted trees and $\T^{\infty}$ be the class of infinite rooted trees with exactly one infinite path starting from the root. This path is called the spine of the tree. If $t \in \T^{\infty}$, we denote by $u_0 = u_0(t), u_1 = u_1(t), ... \in V(t)$ the vertices of the spine (ordered from the root to infinity: $u_0 = r$). From each vertex $u_i$ of the spine grows a finite “bush” $b_i, = b_i(t)$ rooted at $u_i$: it is the induced subgraph of $t$ constituted of the vertices $u$ such that there is a path from $u_i$ to $u$ but not from $u_{i+1}$ to $u$. The bushes are finite rooted trees: $(b_i,u_i) \in \T$.
\end{Def}
\begin{Def}[Tree-indexed random walks]
If $(t,r) \in \T \cup \T^{\infty}$, a SRW indexed by $t$ in dimension $d$ starting from $x \in \Z^d$ is a random function $X_t: V(t) \to \Z^d$ such that $X_t(r) = x$ and $\left(X_t(u)-X_t(v)\right)_{(uv) \in E(t)}$ are i.i.d. increments of a SRW in $\Z^d$.
\end{Def}
\begin{Def}[Galton-Watson trees and Kesten trees]
    Let $\mu: \fix{\NN} \to \fix{[0, \infty)}$ be such that $\sum_{n \geq 0} \mu(n) = 1$ and $\E[\mu] = \sum_{n \geq 0} n\mu(n) = 1$. It is called a critical offspring distribution. We will always assume that $\mu$ has finite variance $\sigma^2 < \infty$. A $\mu$-Galton-Watson ($\mu$-GW) tree is a random tree such that the number of offsprings of each individual are i.i.d. random variables with law $\mu$. We will use the letter $T$ for a $\mu$-GW tree. The size-biased measure $\mu^{\star}$ is defined by $\mu^{\star}(i) = i\mu(i)$ for all $i \geq 0$. It is a probability distribution since $\E[\mu]=1$, it has finite expectation since $\mu$ has finite variance and $\mu^{\star}(0)=0$. A multitype $\mu$-GW tree is a random finite rooted tree such that the number of offsprings of the root is distributed according to $\mu^{\star}-1$ (the root has $i$ offspring with probability $\mu^{\star}(i+1)$) and the number of offsprings of the other vertices are distributed independently according to $\mu$. We will denote by $T^{\star}$ a multitype $\mu$-GW tree. The Kesten tree $(T^{\infty},r)$ associated to $\mu$ is a random element of $\T^{\infty}$, which can be constructed as follows: start with an infinite path of vertices $r = u_0, u_1,...$ called the spine. To each vertex of the spine, graft an independent multitype critical $\mu$-GW tree.
\end{Def}
Kesten showed that this random infinite tree is in some sense the critical $\mu$-GW conditioned to survive, even though this event has probability $0$: see \cite{GWtrees} (Corollary 3.13, 3.14 and 3.15) for a modern introductory reference.
\begin{Rem}
    The SRW indexed by the $\mu$-GW tree $X_T$ is often called the branching random walk (BRW). It is said to be critical if $\mu$ is.
\end{Rem}
\begin{Def}[Infinite snake]\label{def:snake}
    The infinite snake $X_{T^{\infty}}$ is a SRW indexed by the Kesten tree. More precisely, for any $x \in \Z^d$, $(T^{\infty},r)$ is a Kesten tree and conditionally on $T^{\infty}$, the infinite snake started at $x$ $X_{T^{\infty}}$ is a SRW started at $x$ indexed by $T^{\infty}$ and independent of $T^{\infty}$. The infinite snake is said to be recurrent if for all $x \in \Z^d$, $\Proba_x[\exists u \in T^{\infty}, X_u = 0] = 1$, otherwise it is said to be transient. If $X_{T^{\infty}}$ is an infinite snake started at $x$, the spine of the tree $u_0, u_1,...$ performs a SRW in $\Z^d$ started at $x$ that we will denote $S_i = X_{u_i}$ for all $i \geq 0$. Besides, the branches $b_i$ are independent multitype $\mu$-GW trees, so conditionally on $(u_i, S_i)_{i\geq 0}$ the restriction of the infinite snake to each branches are independent and the law $X_{b_i}$ of the infinite snake restricted to $b_i$ is that of a SRW started at $S_i$ indexed by a multitype $\mu$-GW tree.
\end{Def}
In \cite{Curien}, the following is proved:
\begin{Thm}[Benjamini, Curien]\label{thm.curien}
  Assume that $\mu$ is a critical offspring distribution with finite variance. Then the infinite snake is recurrent when $d \leq 4$ and transient when $d \geq 5$.
\end{Thm}
It is proved via unimodularity. We will give another proof by using the tools of KRWs. In our study of the infinite snake conditioned to survive, we will need quantitative estimates of \cite{brw3}, \cite{Zhu}, \cite{ZhuI} that we summarize (and simplify) in the following Theorem. Recall that $X_T$ denotes a SRW on a $\mu$-GW tree.
\begin{Thm}[Le Gall, Lin, Zhu]\label{thm.zhu}\label{thm.quantitative}
    Let $\mu$ be a critical offspring distribution with finite variance. Then,
    \be
        \h(x) = \Proba_x[X_T \text{ reaches }0] \asymp 
        \left\{
            \begin{array}{ll}
                \frac{1}{|x|^2\log(|x|)} \text{ if }d=4\\
                \frac{1}{|x|^{d-2}} \text{ if }d\geq 5\\
                \frac{1}{|x|^2} \text{ if }d \leq 3
            \end{array}
        \right.
    \ee
    where the constants in $\asymp$ depend only on $d$ and $\mu$.
\end{Thm}
In \cite{Curien} (first lines of Section 3.2), it is stated that there would probably be an easier proof of Theorem \ref{thm.curien} if a result like Theorem \ref{thm.zhu} was proved. Such a proof was already given in \cite{ZhuII} but with a different argument. We give such a proof to provide intuition to the reader about the link between killed random walk and infinite snake. Before the proof, we need a last definition:
\begin{Def}[Truncated trees and tree-indexed SRW]
    We will write $T^{\infty}_n$ for the Kesten tree $T^{\infty}$ whose spine is cut at height $n$: the result is an (a.s.) 
finite tree made of the $n$ first vertices of the spine $u_0,...,u_{n-1}$ and the branches grafted to them $b_0,...,b_{n-1}$. Note that the height of this truncated tree can be bigger than $n$ because of the branches.
    We will also use $X_{T^{\infty}_n}$ for the restriction of the infinite snake to $T^{\infty}_n$.
\end{Def}
We now show how Theorems \ref{th.rec} and \ref{thm.quantitative} imply Theorem \ref{thm.curien}.
\begin{proof}
    \fix{We assume that $d \geq 3$ since in dimension $d \leq 2$, the spine alone is a SRW in $\Z^d$, hence it is recurrent so the infinite snake is also recurrent.}
    Using the remarks at the end of Definition \ref{def:snake}:
    \begin{equation}\label{eq.killed}
        \begin{aligned}
            \Proba_x\big[X_{T^{\infty}} \text{ never reaches }0\big] 
            &= \lim_{n \to \infty} \Proba_x\big[X_{T^{\infty}_n} \text{ never reaches }0\big]\\
            &= \lim_{n \to \infty} \E_x\left[\prod_{i=0}^{n-1} \Proba_{S_i}\left[X_{b_i} \text{ never reaches }0\right]\right].
        \end{aligned}
    \end{equation}
    where the expectation concerns the SRW $S_i$ and the probabilities concern the (independent) SRWs indexed by the multitype $\mu$-GW trees $b_i$.  Hence, if we write for all $x \in \Z^d$ 
    \be\label{eq:def:k}
        \K(x) = \Proba_x\left[X_{T^{\star}}\text{ reaches }0\right],
    \ee
    we can make the link between infinite snake and KRW:
    \be
        \Proba_x\big[X_{T^{\infty}} \text{never reaches }0\big] = \lim_{n \to \infty} \E_x\left[\prod_{i=0}^{n-1} (1-\K(S_i))\right] = \Proba_x[S^{\K} \text{ survives}].
    \ee
     Hence the infinite snake is recurrent if and only if the associated KRW $S^{\K}$ dies almost surely, and we just need to check the criterion of Theorem \ref{th.rec} for the particular killing function $\K(x)$. Since for all $i \geq 0$ the number of offspring in the first generation of $b_i$ has law $\mu^{\star}$ (which has finite mean and is a.s. $\geq 1$) and each offspring performs a SRW on a $\mu$-GW tree, Theorem \ref{thm.zhu} and a union bound (for the upper bound) imply for all $x$ and $i \geq 0$
    \be
        \K(x) = \Proba_x\left[X_{b_i}\text{ reaches }0\right] \asymp \h(x).
    \ee
   where the constants in $\asymp$ depend only on $d$ and $\mu$. On the one hand, if $d \geq 5$, using Theorem \ref{thm.zhu} we can find a constant $c$ such that 
    \be
        \K(x) \leq \frac{c}{|x|^{d-2}} = p(|x|).
    \ee
    \fix{The killing function $x \to p(|x|)$ satisfies~\eqref{eq:approx:rot:inv} and
    $
        \int_0^{\infty} rp(r) = \int_0^{\infty} \frac{c}{r^2} < \infty,
    $
    }
    hence the KRW \fix{is not trapped} and the infinite snake is transient.\\
    On the other hand, if $d \leq 4$, using Theorem \ref{thm.zhu} we can find a constant $c$ such that 
    \be
        \K(x) \geq \frac{c}{|x|^2\log(|x|)} = p(|x|).
    \ee
    \fix{The killing function $x \to p(|x|)$ satisfies~\eqref{eq:approx:rot:inv} and
    $
        \int_0^{\infty} rp(r) = \int_0^{\infty} \frac{c}{r\log(r)} = \infty
    $
    }
    so the KRW \fix{is trapped} and the infinite snake is recurrent.
\end{proof}
We now turn to the question that motivates this article: can we condition the infinite snake to \fix{avoid the origin}? In dimension $d \geq 5$, it is just conditioning by an event with positive probability. In dimension $d=4$, the infinite snake exhibits critical behavior by Theorem \ref{thm.zhu}: it is “almost transient” so we can hope to define the infinite snake conditioned to survive with a limiting procedure as in Definition \ref{def:conditioned:KRW}. The more natural way to condition the infinite snake to avoid the origin would be to condition the tree-indexed walk not to hit the origin before height $h$ (on the tree) and to let $h$ go to infinity. 
Intuitively, this limit is complex to identify because conditioning the walk also affects the law of the tree: it favors long and thin trees. Fortunately, there is another way to condition the walk to avoid the origin which makes use of the spinal structure of the Kesten tree and yields an object with a simple geometric description. \fix{Roughly speaking,} the procedure consists in conditioning the spine to hit $\Lambda_R^c$ \fix{“}before\fix{”} any of the branches grafted to the spine hits the origin, and then taking the limit $R \to \infty$. To make it rigorous, we introduce a partial order, \fix{which differs from the “usual” partial order on rooted trees. Recall that $u_0, u_1, \dots$ denotes the spine of the Kesten tree}.
\begin{Def}[Partial order]\label{def:order} If $t \in \T^{\infty}$ is an infinite tree we introduce a partial order on $V(t)$: for $u,v \in b_i, b_j$, $u < v$ if $i<j$. If $x_{t}$ is a (deterministic) realization of the infinite snake ($t \in \T^{\infty}$ and $x:t \to \Z^d$) and $A \subset \Z^d$, we say that “$x_t$ hits $A$ before $0$” if there exists $j \geq 0$ such that $x_{u_j} \in A$ and for all $u \leq u_j$, $x_u \neq 0$.
\end{Def}

\begin{figure}\centering   
	\begin{overpic}[abs,unit=1mm,scale=1]{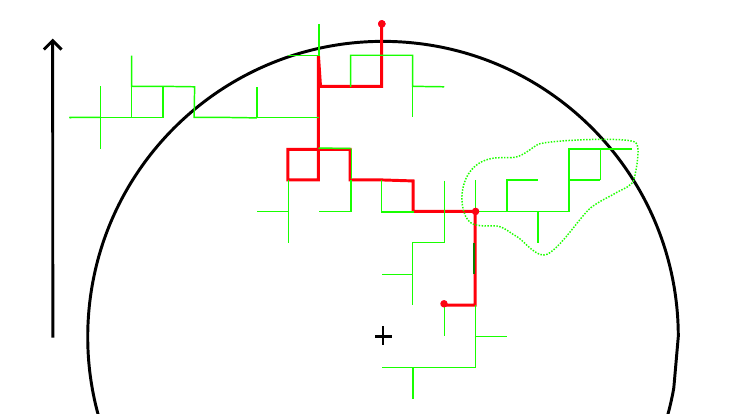}	
        \put(5,40){$R$}
        \put(61,13){$0$}
        \put(70,18){\color{red}$S_0$}
        \put(82,32){\color{red}$S_i$}
        \put(100,30){\color{green}$x_{t_i}$}
        \put(60,68){\color{red}$S_{\tau(B(r)^c)}$}
        \end{overpic}     
        \caption{An example of tree-indexed random walk which escapes the ball $B(R)$ before hitting $0$ in the sense of Definition \ref{def:order}, in dimension $2$. Only the values of the snake truncated at $\tau(B(R)^c)$ are drawn. The values of the SRW performed by the spine are drawn in red, the values of the BRW performed by the branches are in green. Observe that some of the branches can escape $B(R)$, but none of them hits $0$.}
        \label{fig:snake}
\end{figure}
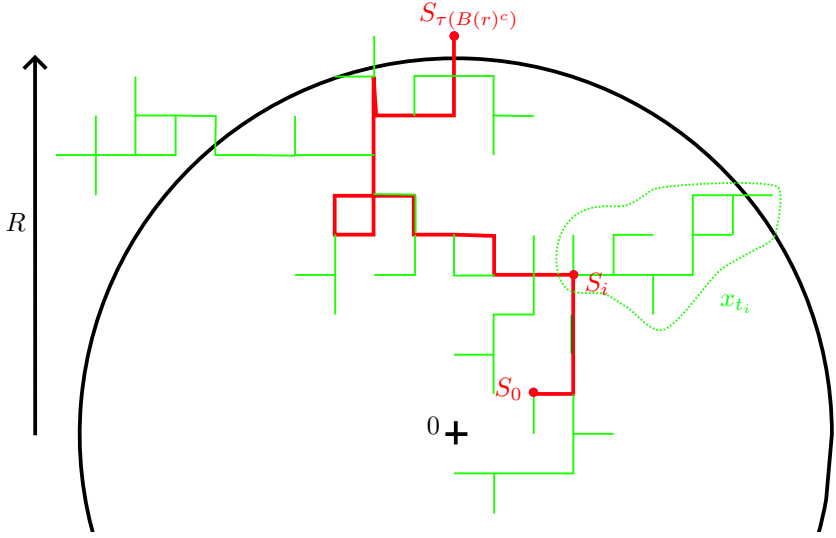
With this definition, we are ready to prove our main theorem on the infinite snake in critical dimension $d=4$.
\begin{Thm}\label{thm:cv:snake}
    Let $\Lambda$ be any exhaustion of $\Z^4$ and $\mu$ be a critical offspring distribution with finite variance. The infinite snake conditioned to hit $\Lambda_R^c$ before $0$ (in the sense of Definition \ref{def:order}) converges in distribution when $R \to \infty$ towards a limiting process which takes values almost surely in $(\Z^4)^{\T^{\infty}}$. More precisely, there exists a function $a: \Z^4 \to \R$ positive on $\Z^4 \setminus \{0\}$ such that for all $n \in \N$, for all $s_0, ..., s_n \in \Z^4 \setminus \{0\}$, for all $t_0, ..., t_{n-1} \in \T$, for all $x_{t_0}: t_0 \to \Z^4 \setminus\{0\}$, ..., $x_{t_{n-1}}: t_{n-1} \to \Z^4 \setminus\{0\}$ with $x_{t_i}(u_i) = s_i$ for all $0 \leq i \leq n-1$:
    \begin{equation}\label{eq:cv:snake}
        \begin{aligned}
            &\Proba_x[\forall 0 \leq i \leq n, S_i = s_i, \forall 0 \leq i \leq n-1, b_i = t_i, X_{b_i} = x_{t_i}|X_{T^{\infty}} \text{ hits }\Lambda_R^c \text{ before }0]\\
            &\overset{R \to \infty}{\longrightarrow} \frac{a(s_n)}{a(s_0)} \Proba[S^{\K}_0 = s_0, ..., S^{\K}_n = s_n] \prod_{i=0}^{n-1} \Proba_{s_i}[X_T = x_{t_i}~|~\text{$X_T$ does not hit }0],
        \end{aligned}
  \end{equation}
  where $S^{\K}$ and $X_T$ are respectively a $4$-dimensional KRW with the killing function of~\eqref{eq:def:k} and a SRW indexed by a multitype $\mu$-GW tree.
\end{Thm}
\begin{Rem}
    Observe that the conditioning appearing in the right-hand side of \eqref{eq:cv:snake} is by an event of positive probability since the multitype critical $\mu$-GW tree $T$ is almost surely finite. In particular the conditional probability could be written explicitly in terms of $\mu$.
\end{Rem}
We now prove this theorem: this is a straightforward consequence of Theorem \ref{thm.cv}, Theorem \ref{thm.zhu} and the interpretation of the infinite snake in terms of KRW used in the proof we gave of Theorem \ref{thm.curien}.
\begin{proof}
    As in \eqref{eq:def:k}, we denote by $\K(x)$ the probability that a SRW indexed by a multitype $\mu$-GW tree started at $x$ ever hits the origin, and $S^{\K}$ the associated KRW. Observe that $\K(x) = 1$ if and only if $x=0$ so $\EK = \Z^4 \setminus \{ 0 \}$ is connected. Recall that $\tau(\Delta)$ denotes the killing time of $S^{\K}$. First of all, 
    \be\label{eq:snake:denumerator}
        \begin{aligned}
            \Proba_x[X_{T^{\infty}} \text{ hits }\Lambda_R^c \text{ before }0] 
            &= \sum_{n \geq 0} \Proba_x[S_n \in \Lambda_R^c,~\forall i < n,~X_{b_i} \text{ does not hit }0]\\
            &= \sum_{n \geq 0} \E_x\bigg[\mathbbm{1}_{S_n \in \Lambda_R^c}\prod_i^{n-1} (1-\K(S_i))\bigg]\\
            &= \Proba_x[\tau(\Lambda_R^c) < \tau(\Delta)].
        \end{aligned}
    \ee
    Hence,
    \be\label{eq:snake:numerator}
        \begin{aligned}
            &\Proba_x[\forall 0 \leq i \leq n, S_i = s_i, \forall 0 \leq i \leq n-1, b_i = t_i, X_{b_i} = x_{t_i} \cap X_{T^{\infty}} \text{ hits }\Lambda_R^c \text{ before }0]\\
            &= \Proba_x[\forall i \leq n, S_i = s_i ]\Bigg(\prod_{i=0}^{n-1}\Proba_{s_i}[X_{T^{\star}}=x_{t_i}]\Bigg)\Proba_{s_n}[X_{T^{\infty}} \text{ hits }\Lambda_R^c \text{ before }0]\\
            &= \Proba_x[\forall i \leq n, S_i = s_i ]\Bigg(\prod_{i=0}^n\Proba_{s_i}[X_{T^{\star}}=x_{t_i}]\Bigg)\Proba_{s_n}[\tau(\Lambda_R^c) < \tau(\Delta)].
        \end{aligned}   
    \ee
    Since in dimension $4$, Theorem \ref{thm.zhu} implies that the killing function $\K$ satisfies the hypothesis of Theorem \ref{thm.cv}, there exists a function $a$ positive on $\Z^4 \setminus \{0\}$ such that for all $x,y \in \Z^d$:
    \be
        \frac{\Proba_y[\tau(\Lambda_R^c) < \tau(\Delta)]}{\Proba_x[\tau(\Lambda_R^c) < \tau(\Delta)]} \overset{R \to \infty}{\longrightarrow} \frac{a(y)}{a(x)}.
    \ee
    Finally, taking ratio of \eqref{eq:snake:numerator} and \eqref{eq:snake:denumerator}:
    \be
        \begin{aligned}
            \Proba_x[\forall 0 \leq i &\leq n, S_i = s_i, \forall 0 \leq i \leq n-1, b_i = t_i, X_{b_i} = x_{t_i}|X_{T^{\infty}} \text{ hits }\Lambda_R^c \text{ before }0]\\
            \overset{R \to \infty}{\longrightarrow} &\frac{a(s_n)}{a(s_0)}\Proba_x[\forall i \leq n, S_i = s_i ]\Bigg(\prod_{i=0}^{n-1}\Proba_{s_i}[X_{T^{\star}}=x_{t_i}]\Bigg)\\
            =&\frac{a(s_n)}{a(s_0)}\Proba_x[\forall i \leq n, S_i = s_i ]\left(\prod_{i=0}^{n-1} (1-\K(s_i))\right)\Bigg(\prod_{i=0}^{n-1}\frac{\Proba_{s_i}[X_{T^{\star}}=x_{t_i}]}{1-\K(s_i)}\Bigg)\\
            =&\frac{a(s_n)}{a(s_0)}\Proba_x[\forall i \leq n, S^{\K}_i = s_i]\prod_{i=0}^{n-1}\Proba_{s_i}[X_{T^{\star}}=x_{t_i}|X_{T^{\star}} \text{ does not reach }0],
        \end{aligned}
    \ee
    finishing the proof.
\end{proof}

\begin{Qu}
    As we have already mentioned just before Definition \ref{def:order}, the conditioning in Theorem \ref{thm:cv:snake} is not the \fix{most} intuitive one. We conjecture that this limiting procedure coincides with any other natural limiting procedure, for example conditioning the tree-indexed SRW not to hit $0$ before height $h$ and letting $h \to \infty$. The infinite snake is the BRW conditioned to be infinite, so in Theorem \ref{thm:cv:snake} we first condition the BRW to be infinite to get the infinite snake and then we condition the infinite snake never to reach $0$. It would be interesting to know if we get the same limit by first conditioning the BRW to avoid the origin (which has positive probability) and then conditioning to be larger and larger.
\end{Qu}
\begin{Qu}
    In dimension $d \leq 3$, there is no critical behavior: the asymptotic decay of the killing function $\K$ given by \ref{thm.zhu} does not enable us to use either Theorem \ref{thm.cv} or Theorem \ref{thm.counterex} so we do not know if the infinite snake conditioned to avoid the origin is well-defined. We conjecture that the limit in distribution in Theorem \ref{thm:cv:snake} exists but that it depends on the exhaustion. 
\end{Qu}


\section{Existence of the conditioned random walk}
\label{proof3d}
This section is dedicated to the proof of Theorem \ref{thm.cv} in dimension $d \geq 3$.

\subsection{Notations and preliminaries}
\fix{Let} $d \geq 3$. Let $(\Lambda_R)_{R \in \N}$ be any exhaustion of $\Z^d$ which will be fixed throughout this section. We try to use as much as possible the notation of \cite{Zhu} since we build on several of their auxiliary lemmas. In what follows, $g(x,y)=g(0,x-y)=g(x-y)$ stands for the usual Green function of the SRW in $\Z^d$\fix{, defined in~\eqref{eq:def:green}}. A path $u_1,..., u_n$ of length $n$ is a sequence of vertices of $\Z^d$ such that $u_i \sim u_{i+1}$ for all $i$. For $u \in \Z^d$ and $A \subset \Z^d$, we will write $\gamma: u \to A$ if $u_1 = u$ and $u_n \in A$. If $\gamma$ is a path of length $|\gamma|$, we will denote by $\s(\gamma)=(2d)^{-|\gamma|}$ its probability weight for the SRW and $\B(\gamma) = \s(\gamma) \prod_{i = 0}^{|\gamma|-1}\big(1-\K(\gamma (i))\big)$ its probability weight with respect to the KRW. The Green function of the KRW can be expressed conveniently with these notations: for $v, w \in \Z^d$, 
\be
    G^{\K}(v,w) = \sum_{n = 0}^{\infty} \Proba_v [S^{\K}_n = w] = \sum_{\gamma: v \to w} \B(\gamma)
\ee
(the sum converges because it is dominated by the Green function of the SRW). We also define the harmonic measure associated to the KRW. 
For $\gamma$ a path, $B \subset \Z^d$, we write $\gamma \sqsubset B$ if all the vertices of $\gamma$ belong to $B$ except maybe the last one. 
Then, we define the killed harmonic measure for $B \subset \Z^d$, $v \in B, w \in B \cup \partial B$,  
\be
    H_B^{\K}(v, w) = \sum_{\gamma: v \to w, \gamma \sqsubset B} \B(\gamma). 
\ee
We will often use the following compact notation: $G^{\K}(v,C) = \sum_{w \in C} G^{\K}(v,w)$ for $v \in \Z^d$ and $C \subset \Z^d$, and similarly for $H_B^{\K}(v, C)$, $G(v,C)$ and $H_B(v,C)$ (the last two refer to the SRW).\par
For all $x \in \Z^d$, $r_1< r_2 \in \R_+$ and $R$ large enough so that $B(r_2) \subset \Lambda_R$, we will split the probability $\Proba_x[\tau(\Lambda_R^c) < \tau(\Delta)]$ in several terms according to $\tau(r_1)$ the first time of exit of $B(r_1)$, $T_2$ the time of last visit to $B(r_2)$ before exiting $\Lambda_R$ and $\tau(\Lambda_R^c)$:
\begin{equation}\label{decomposition}
    \begin{aligned}
        \Proba_x\big[\tau(\Lambda_R^c) < \tau(\Delta)\big]
        &= \sum_{\gamma: x \to \Lambda_R^c, \gamma \sqsubset \Lambda_R} \B(\gamma)\\
        &= \sum_{v \in \partial B(r_1)} \sum_{w \in \partial B(r_2)} H_{B(r_1)}^{\K}(x,v) H_{\Lambda_R}^{\K}(v,w) H_{\Lambda_R \setminus B(r_2)}^{\K}(w, \Lambda_R^c).
    \end{aligned}
\end{equation}
This decomposition is represented on the left of Figure \ref{fig:split}.

\begin{figure}\centering   
	\begin{overpic}[abs,unit=1mm,scale=0.4]{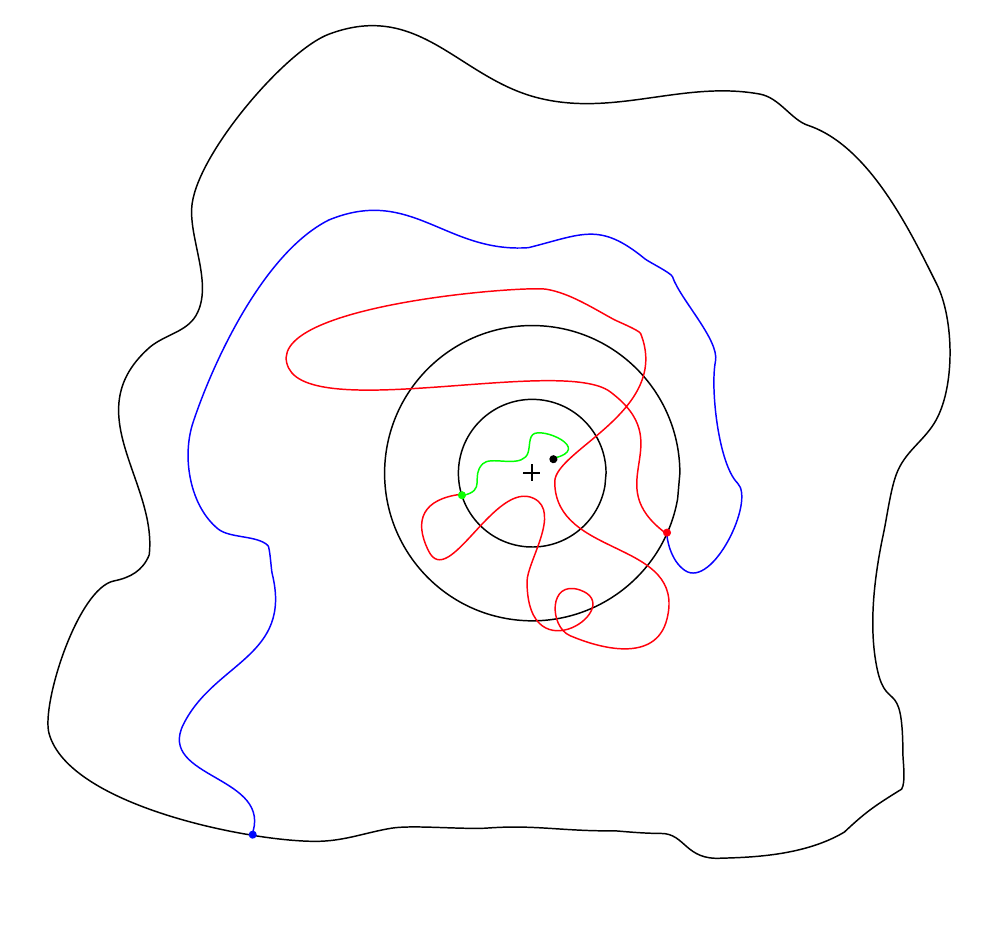}
            \put(23,33){\color{green}$S_{\tau(r_1)}$}
            \put(46,28){\color{red}$S_{T_2}$}
            \put(17,4){\color{blue}$S_{\tau(\Lambda_R^c)}$}
            \put(10,58){$\Lambda_R^c$}
        \end{overpic}     
	\begin{overpic}[abs,unit=1mm,scale=0.4]{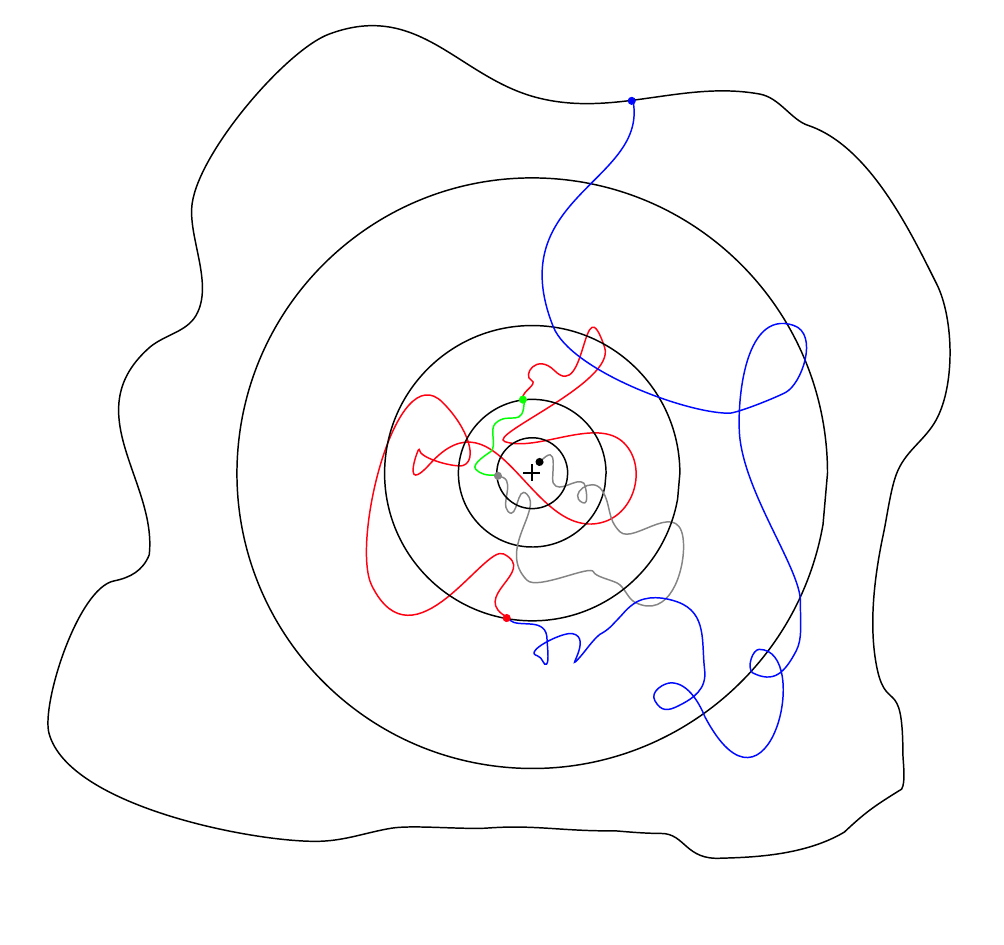}
            \put(29,29){\color{gray}$S_{T_0}$}
            \put(30,38){\color{green}$S_{T_1}$}
            \put(29,20){\color{red}$S_{T_2}$}
            \put(38,60){\color{blue}$S_{\tau(\Lambda_R^c})$}
            \put(10,58){$\Lambda_R^c$}
        \end{overpic}
        \caption{On the left: a path is split as in~\eqref{decomposition}: the concentric balls are $B(r_1) \subset B(r_2)$ (the drawing should be in $d \geq 3$). On the right: a path is split as in~\eqref{decomposition2d}: the concentric balls are $B(r_0) \subset B(r_1) \subset B(r_2) \subset B(r_3)$. Both in $d \geq 3$ and $d=2$, the proof of Theorem~\ref{thm.cv} relies on the fact that if $r_2/r_1$ is large enough, the red portion of the path forgets its starting point (the green bullet).}
        \label{fig:split}
\end{figure}
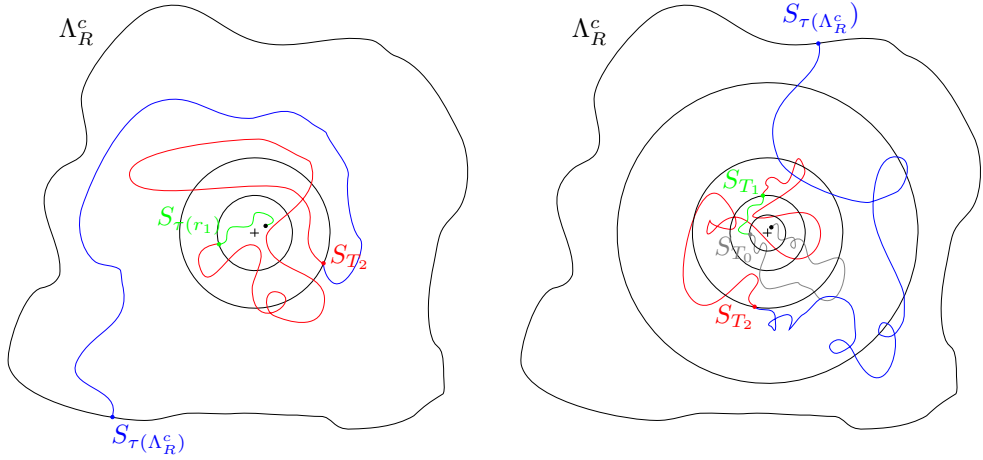

\subsection{Auxiliary results}
We first give a standard estimate on the Green function in dimension $d$, which can be found for example in Theorem 1.5.4 of \cite{Lawler}: 
\begin{Lem}\label{lem:green:estimate}
    There exists a constant $a_d$ depending only on the dimension such that
        \be\label{eq:green:estimate}
            g(u) \overset{|u|\to \infty}{\sim}a_d |u|^{2-d}
        \ee
\end{Lem}
We gather a few definitions and facts from the proof of Lemma 3.1 of \cite{Zhu}. We define for $v, w \in \Z^d$, for $A > 0$ and $0 < \alpha < 0.5$
\begin{equation}
    \left\{
        \begin{array}{ll}
            \Gamma_1(A, v, w) = \{ \gamma: v \to w~|~|\gamma| \geq A |w-v|^2 \}\\
            \Gamma_2(\alpha, v, w) = \{ \gamma: v \to w~|~\gamma \cap B(\alpha |v|) \neq \emptyset \}
        \end{array}
    \right.
\end{equation}
The following holds:
\begin{Lem}\label{lem:zhu}
    \begin{equation}\label{limA}
        \sup_{v, w \in \Z^d} \frac{1}{g(v,w)}\sum_{\gamma \in \Gamma_1(A, v, w)} \s(\gamma) \overset{A \to \infty}{\longrightarrow} 0.
    \end{equation}
    \begin{equation}\label{limalpha}
        \sup_{v, w \in \Z^d, |w| \geq |v|} \frac{1}{g(v,w)}\sum_{\gamma \in \Gamma_2(\alpha, v, w)} \s(\gamma) \overset{\alpha \to 0}{\longrightarrow} 0
    \end{equation}
\end{Lem}
Equation \eqref{limA} is precisely the statement of Lemma 2.3 of \cite{Zhu}. Equation \eqref{limalpha} is almost the same as the statement in page 7 of \cite{Zhu} but we show how to adapt the proof for completeness.
\begin{proof}
    If $|w| \geq |v|$, using~\eqref{eq:green:estimate}, we can find $C>0$ depending only on $d$ such that
    \be
        g(v,w) \geq C|v-w|^{2-d} \geq C2^{2-d}|w|^{2-d}.
    \ee
On the other hand, by the strong Markov property and~\eqref{eq:green:estimate}, there exists $C'>0$ such that 
    \be
        \sum_{\gamma \in \Gamma_2(\alpha, v, w)} \s(\gamma) = \sum_{a \in B(\alpha|v|)}H_{B(\alpha|v|)^c}(v,a)g(a,w) \leq C'|w|^{2-d} \Proba_v[S \text{ hits }B(\alpha|v|)].
    \ee
    Estimating this probability is standard (see for example Proposition 1.5.10 of \cite{Lawler} when $m \to \infty$): there exists $C''>0$ such that
    \be
        \Proba_v[S \text{ hits }B(\alpha|v|)] \leq C''\frac{|v|^{2-d}}{(\alpha|v|)^{2-d}} = C''\alpha^{d-2}.
    \ee
    Note that this also holds when $\alpha|v| < 1$ because then $B(\alpha|v|)$ is empty. Putting these equations together proves Lemma \ref{lem:zhu}.
\end{proof}
The main technical step in the proof of Theorem \ref{thm.cv} will be to show that asymptotically the middle term $H_{\Lambda_R^c}^{\K}(v,w)$ in \eqref{decomposition} only depends on $r_1$ and $r_2$ (when $1 << r_1 << r_2$), and not on $v$ and $w$, which will allow us to factorise the expression into one term depending only on $x, r_1$ and another one depending only on $R, r_2$. 
In view of this, we state a lemma which is inspired by Lemma 3.1 of \cite{Zhu}:
\begin{Lem}\label{lemme3.1}
    Assume that $\K(x) = o(|x|^{-2})$. For all $\varepsilon > 0$, there exists $\rho_1(\eps)>0$ such that for all $r_1 \geq \rho_1(\eps)$, there exists $r_2$ such that for all $v \in \partial B(r_1)$, $w \in \partial B(r_2)$ and $R$ large enough:
    \begin{equation}
        1- \varepsilon \leq \frac{H^{\K}_{\Lambda_R^c}(v,w)}{a_d r_2^{2-d}} \leq 1 + \varepsilon.
    \end{equation}
\end{Lem}
This lemma will be proved in the next subsection.

\subsection{Main proofs}\label{conditionedkrw}
We start by proving Lemma \ref{lemme3.1}.
\begin{proof}
    Recall that $H_{\Lambda_R^c}^{\K}(v, w) = \sum_{\gamma: v \to w, \gamma \sqsubset {\Lambda_R^c}} \B(\gamma)$. 
    The proof will consist in partitioning the paths from $v \to w$ in $\Gamma_1(A,v,w)$, $\Gamma_2(\alpha,v,w)$, and the rest and to show that for well-calibrated $\alpha$, $A$, $r_1$ and $r_2$, the contribution to $g(v,w)$ of the paths of $\Gamma_1(A,v,w)$, $\Gamma_2(\alpha,v,w)$ is negligible, while for the other paths the killing is negligible: $\B(\gamma) \approx \s(\gamma)$.\\
    Let $\eps >0$. Let $0 < \delta <1$. Let $r_1 > 0$ to be determined later and fix $r_2 = \frac{r_1}{\delta}$. By \ref{lem:zhu} we can choose $\alpha = \alpha(\delta) > 0$ and $A = A(\delta) > 0$ such that for all $r_1$ and $r_2 = \delta r_1$, if $v \in \partial B(r_1)$ and $w \in \partial B(r_2)$ ($|v|\leq |w|$ is satisfied since $\delta <1$) then
    \begin{equation}\label{gamma}
        \sum_{\gamma \in \Gamma_1(\alpha,v,w) \cup \Gamma_2(\alpha,v,w)} \s(\gamma) \leq \delta g(v,w)
    \end{equation}
  We are left with controlling the probabilities of the paths that belong neither to $\Gamma_1(A,v,w)$ nor to $\Gamma_2(\alpha, v,w)$. These paths take their steps only in the complementary of $B(\alpha r_1)$, so by the assumption that $\K(|x|) = o(|x|^{-2})$, for all $\eta > 0$, we can take $r_1$ large enough such that for $|x| \geq \alpha r_1$, $\K(x) \leq \frac{\eta}{\alpha^2 r_1^2}$. Now for any $v \in \partial B(r_1)$, $w \in \partial B(r_2)$, if $\gamma: v \to w \notin \Gamma_1(A,v,w) \cup \Gamma_2(\alpha,v,w)$ we have:

  \begin{equation}
        \begin{aligned}
            \frac{\B(\gamma)}{\s(\gamma)} 
            = \prod_{i=0}^{|\gamma|-1}\big(1-\K(\gamma(i))\big)
            \geq \left(1 - \frac{\eta}{\alpha^2 r_1^2}\right)^{|\gamma|}
            &= \exp\left(|\gamma| \ln \left(1 - \frac{\eta}{\alpha^2 r_1^2}\right)\right)\\
            &\geq \exp\left(A|w-v|^2\ln \left(1 - \frac{\eta}{\alpha^2 r_1^2}\right)\right).
        \end{aligned}
  \end{equation}
  where the last inequality comes from the fact that $\gamma \notin \Gamma_1(A, v, w)$. Using that $r_1= \delta r_2 = |v|\leq |w| = r_2$ and $\ln(1-\theta) \geq -2\theta$ for $\theta$ small enough, we get that for $r_1$ large enough,

  \begin{equation}
    A|w-v|^2\ln \left(1 - \frac{\eta}{\alpha^2 r_1^2}\right) \geq 4Ar_2^2\ln \left(1 - \frac{\eta}{\alpha^2 r_1^2}\right) \geq -\frac{8A\eta r_2^2}{\alpha^2r_1^2} \geq -\frac{8A\eta}{\alpha^2\delta^2}.
  \end{equation}
  This is true for all $\eta > 0$ and $r_1$ large enough once $A, \alpha, \delta$ are fixed so we can choose $\eta$ such that $\frac{8A\eta}{\alpha^2\delta^2} \leq \delta$ and now for all $r_1$ large enough, $r_2 = \frac{r_1}{\delta}$, for all $v \in \partial B(r_1)$, $w \in \partial B(r_2)$, $\gamma \notin \Gamma_1(A, v, w) \cup \Gamma_2(\alpha, v, w)$:

  \begin{equation}
    \label{b/s}
    \frac{\B(\gamma)}{\s(\gamma)} \geq \exp(-\delta) \geq 1 - \delta.
  \end{equation}
  Upon increasing $r_1$ we can also assume that for $|x| \geq r_1$,

  \begin{equation}
    \label{controlgreen}
    1-\delta \leq \frac{g(x)}{a_d|x|^{2-d}} \leq 1 + \delta.
  \end{equation}
  We put everything together: for $\delta > 0$, we chose $A = A(\delta)$, $\alpha = \alpha(\delta)$, $\rho_1 = \rho_1(A, \alpha, \delta)$ such that for all $r_1 \geq \rho_1$, if $r_2 = \frac{r_1}{\delta}$, $v \in \partial B(r_1)$, $w \in \partial B(r_2)$, \eqref{gamma} and \eqref{b/s} hold. Upon increasing $\rho_1$, \eqref{controlgreen} also holds. Observe that for $R$ large enough (depending on $r_1$), all the paths $\gamma \notin \Gamma_1(A,v,w)$ do not reach $\Lambda_B(R)^c$, so in particular if $\gamma \notin \Gamma_1(A,v,w) \cup \Gamma_2(\alpha,v,w) := \Gamma_{12}$, $\gamma \subset \Lambda_R$. Hence, for $R$ large enough,

  \begin{equation}\label{lbHK}
        \begin{aligned}
            H_{\Lambda_R}^{\K}(v,w) 
            = \sum_{\gamma: v \to w, \gamma \sqsubset \Lambda_R} \B(\gamma)
            \geq \sum_{\gamma: v \to w \notin \Gamma_{12}} \B(\gamma)
            \overset{\eqref{b/s}}{\geq} (1-\delta) \sum_{\gamma: v \to w \notin \Gamma_{12}} \s(\gamma)
            &\overset{\eqref{gamma}}{\geq} (1-\delta)^2g(v,w)\\
            &\geq \frac{a_d(1-\delta)^3}{(1+\delta)^{d-2}r_2^{d-2}}
    \end{aligned}
  \end{equation}
  where the last line is justified by \eqref{controlgreen} and the choice of $r_1 = \delta r_2$ so that $|w-v| \leq (1+\delta)r_2$. This is the lower bound of Lemma \ref{lemme3.1} if we choose $\delta = \delta(\eps)$ small enough: the choice of $\rho_1$ depends only on $\eps$ (via $\delta, A(\delta), \alpha(\delta)$), and then \eqref{hbHK} holds for all $r_1 \geq \rho_1(\eps)$, $r_2 = \frac{r_1}{\delta}$ and $R$ large enough compared to $r_1$.\\
  The upper bound is easier: for the same choice of $\delta, r_1, r_2 = \frac{r_1}{\delta}$ we can write for $v \in \partial B(r_1)$ and $w \in \partial B(r_2)$,
  \begin{equation}\label{hbHK}
        H_{\Lambda_R}^{\K}(v,w) 
        \leq g(v,w)
        \overset{\eqref{controlgreen}}{\leq} (1+\delta)a_d |w-v|^{2-d}
        \leq (1+\delta)(1-\delta)^{2-d}r_2^{2-d}
  \end{equation}
  by \fix{the triangle} inequality. This is the upper bound of Lemma \ref{lemme3.1} if $\delta = \delta(\eps)$ is small enough.
\end{proof}
Once this lemma is proved, the proof of Theorem \ref{thm.cv} is relatively straightforward.
\begin{proof}
    We start by proving the existence of the limit \eqref{prop.cv}. Let $x,y \in \EK$. Let $\varepsilon > 0$. Choose $\rho_1 = \rho_1(\eps)$ associated to $\varepsilon$ by Lemma \ref{lemme3.1}, and $r_1 \geq \max(\rho_1, |x|, |y|)$. Let $r_2$ and $R_0$ be associated to $r_1$ also by Lemma \ref{lemme3.1} such that \fix{for all} $v \in B(r_1)$, \fix{for all} $w \in B(r_2)$, \fix{for all} $R \geq R_0$,
    \begin{equation}\label{controlHK}
        1-\varepsilon \leq \frac{H_{\Lambda_R}^{\K}(v,w)}{g(v,w)} \leq 1 + \varepsilon.
    \end{equation}
    Upon increasing $R_0$, we can assume that for all $R \geq R_0$, $B(r_2) \subset \Lambda_R$ (by definition of an exhaustion). By using this inequality in \eqref{decomposition} (which was true for arbitrary $r_1 < r_2$ and $R$ such that $B(r_2) \subset \Lambda_R$), we now get, \fix{for all} $R > R_0$:
    \begin{equation}\label{eq.x.lb}
        \Proba_x[\tau(\Lambda_R^c) < \tau(\Delta)] \geq (1-\epsilon) a_d r_2^{2-d}  \left(\sum_{v \in \partial B(r_1)} H_{B(r_1)}^{\K}(x,v) \right) \left( \sum_{w \in \partial B(r_2)} H_{\Lambda_R \setminus B(r_2)}^{\K}(w, \partial \Lambda) \right)
    \end{equation}
    and
    \begin{equation}\label{eq.x}
      \Proba_x[\tau(\Lambda_R^c) < \tau(\Delta)] \leq (1+\epsilon) a_d r_2^{2-d}  \left(\sum_{v \in \partial B(r_1)} H_{B(r_1)}^{\K}(x,v) \right) \left( \sum_{w \in \partial B(r_2)} H_{\Lambda_R \setminus B(r_2)}^{\K}(w, \Lambda_R^c) \right).
    \end{equation}
    The same holds for $y$:
    \begin{equation}\label{eq.y}
        \Proba_y[\tau(\Lambda_R^c) < \tau(\Delta)] \leq (1+\epsilon) a_d r_2^{2-d}\left(\sum_{v \in \partial B(r_1)} H_{B(r_1)}^{\K}(x,v) \right) \left( \sum_{w \in \partial B(r_2)} H_{\Lambda_R \setminus B(r_2)}^{\K}(w, \Lambda_R^c) \right)
    \end{equation}
    We take the quotient of \eqref{eq.y} by \eqref{eq.x.lb} and  (the denominator is $>0$ because $x \in \EK$) and get \fix{for all} $R > R_0$:
    \begin{equation}
        \frac{\Proba_y[\tau(\Lambda_R^c) < \tau(\Delta)]}{\Proba_x[\tau(\Lambda_R^c) < \tau(\Delta)]} \leq \frac{1+\epsilon}{1-\epsilon} \frac{\sum_{v \in \partial B(r_1)} H_{B(r_1)}^{\K}(x,v)}{\sum_{v \in \partial B(r_1)} H_{B(r_1)}^{\K}(y,v)}.
    \end{equation}
    The same holds if we \fix{exchange} $x$ and $y$ (and take the inverse):
    \begin{equation}
        \frac{1-\epsilon}{1+\epsilon} \frac{\sum_{v \in \partial B(r_1)} H_{B(r_1)}^{\K}(x,v)}{\sum_{v \in \partial B(r_1)} H_{B(r_1)}^{\K}(y,v)} \leq \frac{\Proba_y[\tau(\Lambda_R^c) < \tau(\Delta)]}{\Proba_x[\tau(\Lambda_R^c) < \tau(\Delta)]}
    \end{equation}
    Recall that $x$ and $y$ are fixed. By taking \fix{logarithm} of these inequalities (which is possible since $x \in \EK$), we get that the sequence $u_R(x,y) = \ln \left(\frac{\Proba_y[\tau(\Lambda_R^c) < \tau(\Delta)]}{\Proba_x[\tau(\Lambda_R^c) < \tau(\Delta)]}\right)$ is a Cauchy sequence in $R$: for all $\varepsilon > 0$ there exists $R_0$ such that for all $R, R' > R_0$, $|u_R(x,y) - u_{R'}(x,y)| \leq \varepsilon$. Hence the sequence converges when $R \to \infty$. Besides, the upper and lower bound do not depend on the specific exhaustion, so the limit is independent of the exhaustion, which gives \eqref{prop.cv}.\\
    The rest of the theorem easily follows: for a fixed $x_0 \in \EK$, we can define for all $x \in \EK$:
  \begin{equation}
    \label{def.a}
    a(x) = \lim_{R \to \infty} \frac{\Proba_x[\tau(\Lambda_R^c) < \tau(\Delta)]}{\Proba_{x_0}[\tau(\Lambda_R^c) < \tau(\Delta)]}.
  \end{equation}
  If we define $a(x) =0$ for $x \in \EK^c$, \eqref{def.a} also holds for $x \in \EK^c$ since the quotient on the right-hand side is $0$ for $R$ large enough. This function is massive harmonic as a pointwise limit of massive harmonic functions, which concludes the proof of Theorem \ref{thm.cv}.
\end{proof}


\section{Counter-example}\label{counterex}
\fix{In this section, we introduce rigorously the killed Browinan motion (KBM) and we prove Theorem~\ref{thm.counterex}, which provides an example where the KBM conditioned to survive is not well-defined.}
In this section, $a, c >0$ are absolute constants tha\fix{t} can change from one line to the other. Also, in this section only, $B(r)$ \fix{and $\partial B(r)$} denote the \fix{E}uclidean ball of $\R^2$ \fix{and its boundary} (and not \fix{their} intersection\fix{s} with $\Z^d$), and for $A \subset \R^d$, $A^c := \R^d \setminus A$. 
\subsection{Preliminaries on the killed \fix{B}rownian motion}
\label{defkbm}
In this section we develop the continuous counterpart of the KRW by considering a Brownian particle evolving in a field of traps in $\R^d$ instead of a discrete particle evolving according to a random walk in $\Z^d$. 
The parameter of the model is again a killing function, $\K: \R^d \to \R_+$, which we will assume here to be continuous and bounded. This killing function gives the rate at which the Brownian particle dies depending on its position. 
Since we are typically interested in the behavior of the killed \fix{B}rownian motion (KBM) when the killing vanishes at infinity, we will assume for convenience that $\K$ is bounded.\par

We mentioned a possible approach to define the KBM in the introduction of the paper. It consists in defining the KBM as a (sub)diffusion process on $\R^d$ with generator $A$ defined by \eqref{def:generator:KBM}. This defines a sub\fix{-M}arkovian process which can be naturally extended to a Markovian process on $\R^d \cup {\Delta}$ where $\Delta$ is an additional cemetery state and $\R^d \cup{\Delta}$ is endowed with a metric that makes it compact. This standard procedure is called the minimal extension.\par

We will use a more elementary definition of the KBM: it is the same as the one developed in Section 2.3 of \cite{Sznitman} but in our \fix{simpler} setting we can avoid \fix{many} technicalities. We let $B$ be a standard Brownian motion in $\R^d$ and $\xi$ be an independent exponential variable with law $\exp(-\xi)d\xi$ on $\R_+$ defined on the same probability space $(\Omega, \F, \Proba)$. Let $(\F_t)_{t \geq 0}$ be the natural filtration of the Brownian motion. 
We define the killing time as 
\be
    \tau(\Delta) = \tau(\Delta)(B,\xi) = \inf \Bigg\{t \geq 0 \bigg| \int_0^t \K(B_s)ds \geq \xi\Bigg\}.
\ee
It is a stopping time. The independence of $B$ and $\xi$ and the exponential distribution of $\xi$ imply the following property:
\begin{Prop}\label{prop:kbm}
    For all $f$ real-valued measurable function on the Wiener space and $\tau$ stopping time adapted to $(\F_t)_{\fix{t \geq 0}}$,
    \be
        \E_x\big[f(B)\mathbbm{1}_{\tau \leq \tau(\Delta)}\big] = \E_x\Bigg[f(B) \exp\bigg(-\int_0^{\tau}\K(B_s)ds\bigg)\Bigg].
    \ee
\end{Prop}
Observe that the expectation on the right-hand side only involves the law of the Brownian motion $B$. This proposition justifies the analogy with the KRW: the product of killing probabilit\fix{ies} in the discrete is replaced by the exponential of the integral of the killing rate in the continuum.\par
We introduce some definitions and notations related to our problem. If $A \subset \R^2$, we will denote by $\tau(A)$ the hitting time of $A$ by $B$ by $\tau(r) = \tau(B(r))$.
\begin{Def}
    We say that an increasing sequence $(\Lambda_R)_{R \in \N}$ of strict bounded subsets of $\R^2$ is an exhaustion of $\R^2$ if for all $r >0$, there exists $R \in \N$ such that  $B(r) \subset \Lambda_R$.
\end{Def}
We finally recall some classical results on the Brownian motion. If $W$ is a one dimensional Brownian motion, $x >0$ and $\sigma(x)$ denotes the hitting times of $x$, using the reflection principle and rescaling property of the Brownian motion:
\be\label{eq:LDP1d}
    \Proba_0[\sigma(x) \leq t] = \Proba_0\left[\sup_{[0,t]}W \geq x\right] =  \Proba_0[|W_t| \geq x] = \Proba_0\left[|W_1|\geq \frac{x}{\sqrt{t}}\right] \leq \exp\left(-\frac{x^2}{2t}\right).
\ee
A union bound implies the following large deviation estimate for the one-dimensional BM:
\be\label{eq:LDP:twosided}
    \Proba_0[\min(\sigma(x),\sigma(-x)) \leq t] \leq 2\exp\left(-\frac{x^2}{2t}\right).
\ee
Since the coordinates of the 2-dimensional BM are 1-dimensional BM\fix{s} and $B_{\R}(r/2) \times B_{\R}(r/2) \subset B_{\R^2}(r)$, a union bound implies for all $x \in \R^2$, $t,r \geq 0$
    \be\label{eq:2d:LDP}
        \Proba_x[\tau(B(x,r)^c) \leq t] \leq 4\exp\left(-\frac{r^2}{8t}\right).
    \ee
    We also recall the classical skew-product decomposition of the BM in two dimensions, see for example Section 15.7 of the book \cite{ItoMcKean}.
\begin{Def}
    For $x,y \in \R^2 \setminus \{0\}$, we will denote by $\arg(x)$ the argument of $x$ $mod~2\pi$. When we talk about a difference of argument $|\arg(x)-\arg(y)|$, we mean the length of the minimal arc between $\frac{x}{|x|}$ and $\frac{y}{|y|}$ on the unit circle.
\end{Def}
\begin{Prop}
  \label{skew}
  Let $x \in \R^2 \setminus \{ 0 \}$ be fixed and $r = |x|$. Let $(\beta_t)_{t \in \R_+}$ be a Bessel process of dimension $2$ (i.e. it is the norm of a two-dimensional BM) started at $r$. Define the associated Bessel clock
  \begin{equation}
    R_t = \int_{s=0}^t \frac{ds}{\beta_s^2}.
  \end{equation}
  Let $W$ be a one-dimensional BM starting at $0$ independ\fix{e}nt from $\beta$. Then, $\big(\beta_t e^{i(\arg(x)+W_{R_t})}\big)_{t\in \R_+}$ is a two-dimensional BM starting from $x$.
\end{Prop}

\subsection{Proof of the \fix{theorem}}
In this section we state and prove a series of auxiliary results and show how they imply \fix{Theorem~\ref{thm.counterex}}.
\paragraph{Overview of the proof}
The core of the proof will be to understand the event $\fix{\{}\tau(R) < \tau(\Delta)\fix{\}}$. We will first give matching upper and lower bound for \fix{the probability of} this event by combining elementary large deviation estimates and bounds on the killing time. Refining these estimates, we will be able to show that with high probability, on the event $\fix{\{}\tau(R) < \tau(\Delta)\fix{\}}$, \fix{the KBM's trajectory is close to a straight line and its speed is close to a constant (depending on $\alpha$)}. We will rule out all other possibilities: if the BM is too slow or deviates \fix{too much} from this trajectory, it cannot survive, while the probability that it goes too fast is control\fix{l}ed by large deviations estimates. Once this event is well understood, the theorem will be easy to prove: if the only way for the KBM to survive is to escape the origin \fix{close to} a straight line, if the KBM starts from (say) $(r,0)$ conditioning the KBM to escape to the right (its natural behavior) will not have the same effect as conditioning it to escape to the left (very unlikely behavior).\par \bigskip
The first step is to observe that since the killing function, the BM and the \fix{E}uclidean ball $B(r)$ are \fix{rotationally} invariant, the law of the hitting time of $B(r)^c$ starting from $x$ only depends on $r=|x|$. Hence \fix{with a small abuse of notation} we can write
\begin{equation}\label{eq:abuse}
  \Q_r[\tau(r) < \tau(\Delta)] := \Q_x[\tau(r) < \tau(\Delta)]
\end{equation}
for the common law. This argument will be used repeatedly and not be explicited every time. Our key analysis will take place on a small space scale and then extend to the full picture, using the following decomposition: for all $r>0$, $n \in \N$, by Proposition \ref{prop:kbm} and $n$ consecutive uses of the Strong Markov property,
\begin{equation}
  \label{eq.rtoR}
  \Q_r[\tau(2^nr) < \tau(\Delta)] = \prod_{i=0}^{n-1} \Q_{2^ir}\big[\tau(2^{i+1}r) < \tau(\Delta)\big].
\end{equation}
We start by giving estimates for each term of this product. These lemmas are interesting in themselves because they show the balance between the killing and the entropy. Define $\beta > 0$ such that
\begin{equation}
  \label{defepsilon}
  2 - \alpha - 2\beta = 2\beta,
\end{equation}
namely $\beta = \frac{2-\alpha}{4}$. We will see in the proofs the reason for this choice. We first give a lower bound:

\begin{Lem}
  \label{lbrto2r}
  There exist constants $c,a > 0$ such that for all $r \geq 1$,
  \begin{equation}
    \Q_r\big[\tau(2r) < \tau(\Delta)\big] \leq c\exp(-ar^{2\beta})
  \end{equation}
\end{Lem}
\begin{proof}
    The proof consists in finding the good time scale at which the killing and the entropy compensate. By union bound,
    \begin{equation}
        \Proba_r\big[\tau(2r) 
        < \tau(\Delta)\big] \leq \Proba_r\big[\tau(2r) < r^{2-2\beta}\big]
        + \Proba_r\Big[\tau(2r) < \tau(\Delta) , \tau(2r) > r^{2-2\beta}\Big].
    \end{equation}
    On the one hand, if we pick any $x$ with $|x|=r$, \fix{using the notation~\eqref{eq:abuse} we can write}
    \be
        \Proba_r\big[\tau(2r) < r^{2-2\beta}\big] 
        = \Proba_x\big[\tau(2r) < r^{2-2\beta}\big]
        \leq \Proba_x\big[\tau(B(x,r)^c) < r^{2-2\beta}\big]
        \leq 4\exp\left(-\frac{r^{2\beta}}{8}\right)
    \ee
    by~\eqref{eq:2d:LDP}. On the other hand, since on the event $\fix{\big\{}\tau(2r) < \tau(\Delta)\fix{\big\}} \cap \fix{\big\{}\tau(2r) \geq r^{2-2\beta}\fix{\big\}}$ the killing is at least $\K(2r) = \frac{1}{(2r)^{\alpha}}$ and the KBM survives for a time at least $r^{2-2\beta}$, we obtain
  \begin{equation}
    \begin{aligned}
      \Proba_r\Big[\tau(2r) < \tau(\Delta) , \tau(2r) \geq r^{2-2\beta}\Big]
      &= \E_r\Bigg[\exp\bigg(-\int_0^{\tau(2r)}\K(B_s)ds\bigg)\mathds{1}_{\tau(2r) \geq r^{2-2\beta}}\Bigg] \\
      &\leq \E_r\Bigg[\exp\bigg(-\int_0^{r^{2-2\beta}}\K(2r)ds\bigg)\Bigg]\\
      &\leq \exp(-2^{-\alpha}r^{2-2\beta-\alpha})
      = \exp(-2^{-\alpha}r^{2\beta}).
    \end{aligned}
  \end{equation}
  The last equality comes from (and actually is the reason for) the definition of $\beta$ in~\eqref{defepsilon}. Putting the last three equations together proves the lemma.
\end{proof}

Using this upper bound, we provide a matching lower bound and give a more precise description of the event $\fix{\{}\tau(2r) < \tau(\Delta)\fix{\}}$.
\begin{Lem}\label{rto2r}
    For all $T \geq 1$, define
    \begin{equation}\label{defAr}
        A_r = A_r(T) := \fix{\big\{}\tau(2r) < Tr^{2-2\beta}\fix{\big\}} \cap \fix{\big\{}\tau(2r) < \tau(r/2)\fix{\big\}},
    \end{equation}
    There exist some universal constants (not the same as in the preceding lemma) $c,c', a, a' > 0$ and a fixed $T_0 \geq 1$ such that for all $r \geq 1$, denoting $A_r = A_r(T_0)$,
    \begin{equation}\label{eq:lower:bound:KBM}
         \Q_r\fix{\big[}\fix{\big\{}\tau(2r) < \tau(\Delta)\fix{\big\}} \cap A_r\fix{\big]} \geq c\exp(-ar^{2\beta})
    \end{equation}
    and
    \begin{equation}
        \Q_r\big[\fix{\big\{}\tau(2r) < \tau(\Delta)\fix{\big\}} \setminus A_r\big]
        \leq c'\exp(-a'r^{2\beta}) \Q_r[\tau(2r) < \tau(\Delta)].
    \end{equation}
\end{Lem}
\begin{Rem}\label{rem:schilder}
  Actually, by a cautious use of Schilder's theorem (see for example Theorem 5.2.3 of \cite{LDP}), we could identify the rate of exponential decay, namely
  \begin{equation}
    \frac{1}{r^{2\beta}} \log \Q_r[\tau(2r) < \tau(\Delta)] \to m
  \end{equation}
  where $m$ is the minimum of a modified version of the rate function in Schilder's Theorem including an additional term involving the scaling limit of the killing function. We will not need such precise estimates here, which would not even simplify our proof.
\end{Rem}
\begin{proof}
    Let $T \geq 1$ to be determined in the proof (it will depend only on $\alpha$) and recall the definition of $A_r$ \fix{from}~\eqref{defAr}. 
    We first prove the lower bound, using that on the event $\fix{\big\{}\tau(2r) < \tau(\Delta)\fix{\big\}} \cap A_r$, the killing is (deterministically) not too heavy: since $T \geq 1$,
    \begin{equation}\label{Ar}
        \begin{aligned}
            \Proba_r\big[\fix{\big\{}\tau(2r) < \tau(\Delta)\fix{\big\}} \cap A_r\big] 
            &\geq \Proba_r\bigg[\tau(2r) < \tau(\Delta) , \tau(2r) < r^{2-2\beta} , \tau(2r) < \tau(r/2)\bigg]\\
            &= \E_r\Bigg[\exp\bigg(-\int_0^{\tau(2r)}\K(B_s)ds\bigg)\mathbbm{1}_{\fix{\{}\tau(2r) < r^{2-2\beta}\fix{\}} \cap \fix{\{}\tau(2r) < \tau(r/2)\fix{\}}}\Bigg]\\
            &\geq \exp\big(-r^{2-2\beta}\K(r/2)\big)\Proba_r\Big[\tau(2r) < r^{2-2\beta} , \tau(2r) < \tau(r/2)\Big]\\
            &\geq \exp\big(-2^{\alpha}r^{2\beta}\big)\Proba_r\Big[\tau(2r) < r^{2-2\beta} , \tau(2r) < \tau(r/2)\Big]
        \end{aligned}
  \end{equation}
  Controlling this probability is elementary: by the scaling property of the BM,
  \be
        \Proba_r\Big[\tau(2r) < r^{2-2\beta} , \tau(2r) < \tau(r/2)\Big] = \Proba_1\Big[\tau(2) < r^{-2\beta} , \tau(2) < \tau(1/2)\Big].
  \ee
  Writing $W^x$ and $W^y$ for the independent horizontal and vertical components of a BM \fix{and using the notation~\eqref{eq:abuse}},
    \be
        \begin{aligned}
            \Proba_{(1,0)}\Big[\tau(2) < r^{-2\beta} , \tau(2) < \tau(1/2)\Big] 
            &\geq \Proba^x_1\big[\sigma(1/2) \geq r^{-2\beta}\big]\Proba^y_0\Big[|W^y_{r^{-2b}}|\geq 2\Big]\\
            & \geq \Proba^x_1\big[\sigma(1/2) \geq 1\big]\Proba^y_0\Big[|W^y_1|\geq 2r^{\beta}\Big]\\
            &\geq c\exp(-ar^{-2\beta})
        \end{aligned}
    \ee
    for some constants $c, a > 0$ by using $r \geq 1$ and a basic estimate on the Gaussian distribution (and observing that $\Proba^x_1\big[\sigma(1/2) \geq 1\big]$ is a positive constant). Combining this estimate with~\eqref{Ar} \fix{gives}~\eqref{eq:lower:bound:KBM}, which is the first statement of the lemma.\par
    In the second statement, we want to show that $A_r$ is (with high probability) the only way for $\fix{\big\{}\tau(2r)<\tau(\Delta)\fix{\big\}}$ to happen. If we define $B_r = \fix{\big\{}\tau(r/2)<\tau(2r)\fix{\big\}}$ and $C_r = \fix{\big\{}\tau(2r) > Tr^{2-2\beta}\fix{\big\}}$ then
    \begin{equation}\label{ArBrCr}
            \fix{\big\{}\tau(2r) < \tau(\Delta)\fix{\big\}} \setminus A_r = \fix{\big\{}\tau(2r) < \tau(\Delta)\fix{\big\}} \cap \big\{B_r \cup C_r\big\}.
    \end{equation}
    On the one hand, using the strong Markov property twice and Lemma \ref{lbrto2r},
    \begin{equation}
        \begin{aligned}
            \Proba_r\fix{\big[}\fix{\big\{}\tau(2r) < \tau(\Delta)\fix{\big\}} \cap B_r\fix{\big]} &= \Proba_r\big[\tau(r/2) < \tau(2r) < \tau(\Delta)\big]\\
            &= \Proba_r\big[\tau(r/2) < \tau(2r) \wedge \tau(\Delta)\big]\Proba_{r/2}\big[\tau(r) < \tau(\Delta)] \Proba_r[\tau(2r)<\tau(\Delta)\big]\\
            &\leq \Proba_{r/2}\big[\tau(r) < \tau(\Delta)] \Proba_r[\tau(2r)<\tau(\Delta)\big]\\
            &\overset{\ref{lbrto2r}}{\leq} c\exp(-ar^{-2\beta})\Proba_r\big[\tau(2r)<\tau(\Delta)\big].
        \end{aligned}
    \end{equation}
    On the other hand, on $C_r$, the killing is at least $\K(2r)$ for a time at least $Tr^{2-2\beta}$, so
    \begin{equation}
        \Proba_r\fix{\big[}\fix{\big\{}\tau(2r) < \tau(\Delta)\fix{\big\}} \cap C_r\fix{\big]} \leq \exp\big(-Tr^{2-2\beta}\K(r/2)\big) = \exp(-2^{-\alpha}Tr^{2\beta}).
    \end{equation}
    A union bound and the last two equations imply
    \begin{equation}\label{BrCr}
        \begin{aligned}
            \Proba_r\big[\fix{\big\{}\tau(2r) < \tau(\Delta)\fix{\big\}} \cap \big\{B_r \cup C_r\big\}\big] &\leq \Proba_r\big[\tau(2r) \leq T r^{2-2\beta}\big] + \Proba_r\big[\tau(r/2) < \tau(2r) < \tau(\Delta)\big]\\
            &\leq \exp(-2^{-\alpha}Tr^{2\beta}) + c\exp(-ar/2)\Proba_r\big[\tau(2r)<\tau(\Delta)\big].
        \end{aligned}
    \end{equation}
    Using this equation and the lower bound~\eqref{eq:lower:bound:KBM}, we can fix $T_0$ large enough (depending only on $\beta$ and some universal constants) and some constants $c, a > 0$ such that for all $r \geq 1$,
    \begin{equation}
        \Proba_r\big[\fix{\big\{}\tau(2r) < \tau(\Delta)\fix{\big\}} \cap \big\{B_r \cup C_r\big\}\big] \leq c\exp(-ar^{2\beta})\Proba_r\big[\tau(2r) < \tau(\Delta)\big].
    \end{equation}
    Taking probabilities in~\eqref{ArBrCr} and using the last equation gives the second statement of the Lemma.
\end{proof}

\begin{Rem}
    The last two lemmas could be easily proved by using more involved large deviation estimates at time $r^{2-2\beta}$ combined with estimates on the killing function (see Remark \ref{rem:schilder}), but we chose to use only elementary estimates. The next lemma is more delicate since it involves different length scales for the radial and tangential components of the BM.
\end{Rem}

\fix{Elaborating on this lemma, we can show additionally that on the event of survival from $r$ to $2r$, the KBM “does not wind around the origin” with high probability: intuitively, the trajectories that wind around the origin are longer hence less likely to survive.
Actually, we will only need that the arguments of the two endpoints are close, so we only state the following (weaker) result.}
\begin{Lem}\label{Arprime}
    There exist some universal constants (not the same as in the preceding lemma) $c, a > 0$ such that for all $r \geq 1$,
    \begin{equation}
        \Q_r\big[\tau(2r)<\tau(\Delta) , |\arg(B_{\tau(2r)})-\arg(B_0)|\geq r^{-\beta/2}\big] \leq c\exp(-ar^{\beta})\Q_r\big[\tau(2r)<\tau(\Delta)\big]
    \end{equation}
\end{Lem}
\begin{proof}
    Define for all $T \geq 1$, $r \geq 1$,
    \begin{equation}\label{defArprime}
        A'_r = A_r \cap \fix{\Big\{}\big|\arg(B_{\tau(2r)})-\arg(B_0)\big| \leq r^{-\beta/2}\fix{\Big\}}.
    \end{equation}
    \fix{Choose} $T_0 \geq 1$ and $a,a',c,c' > 0$ as in Lemma \ref{rto2r}, and let $r \geq 1$. First of all, by union bound,
    \begin{equation}\label{unionArprime}
        \begin{aligned}
            \Proba_r\big[\fix{\big\{}\tau(2r)<\tau(\Delta)\fix{\big\}} \setminus A'_r\big]  
            &\leq \Proba_r\big[\fix{\big\{}\tau(2r)<\tau(\Delta)\fix{\big\}} \setminus A_r\big]\\
            &+ \Proba_r\Big[\fix{\big\{}\tau(2r)<\tau(\Delta)\fix{\big\}} \cap A_r \cap \fix{\big\{}\big|\arg(B_{\tau(2r)})-\arg(B_0)\big| \geq r^{-\beta/2}\fix{\big\}}\Big].
        \end{aligned}
    \end{equation}
    The first term is bounded by Lemma \ref{rto2r}:
    \be\label{eq:1st}
        \Proba_r\big[\fix{\big\{}\tau(2r)<\tau(\Delta)\fix{\big\}} \setminus A_r\big] \leq c\exp(-ar^{2\beta})\Proba_r\big[\tau(2r)<\tau(\Delta)\big].
    \ee 
    We now bound the other term by using the skew product decomposition Proposition \ref{skew}. On $A_r$,
    \begin{equation}
        R_{\tau(2r)} = \int_0^{\tau(2r)}\frac{ds}{\beta_s^2} \leq \int_0^{Tr^{2-2\beta}}\frac{ds}{(r/2)^2} \leq 4T_0r^{-2\beta}.
    \end{equation} 
    Hence,
    \be
        A_r \cap \fix{\big\{}\big|\arg(B_{\tau(2r)})-\arg(B_0)\big| \geq r^{-\beta/2}\fix{\big\}} \subset \fix{\bigg\{}\sup_{[0,4Tr^{-2\beta}]}|W_t| \geq r^{-\beta/2}\fix{\bigg\}}
    \ee
    So by independence of the radial and tangential components in the skew-product decomposition
    \be\label{eq:2nd}
        \begin{aligned}
            \Proba_r\Big[\fix{\big\{}\tau(2r)<\tau(\Delta)\fix{\big\}}& \cap A_r \cap \fix{\big\{}\big|\arg(B_{\tau(2r)})-\arg(B_0)\big| \geq r^{-\beta/2}\fix{\big\}}\Big]\\
            &\leq \Proba_r\bigg[\fix{\big\{}\tau(2r)<\tau(\Delta)\fix{\big\}} \cap A_r \cap \fix{\bigg\{}\sup_{[0,4Tr^{-2\beta}]}|W_t| \geq r^{-\beta/2}\fix{\bigg\}}\bigg]\\
            &\leq \Proba_r\Big[\fix{\big\{}\tau(2r)<\tau(\Delta)\fix{\big\}} \cap A_r\Big]\Proba\bigg[\sup_{[0,4Tr^{-2\beta}]}|W_t| \geq r^{-\beta/2}\bigg]\\
            &\overset{\eqref{eq:LDP:twosided}}{\leq} 2\exp\left(-\frac{r^{\beta}}{8T_0}\right)\Proba_r\Big[\fix{\big\{}\tau(2r)<\tau(\Delta)\fix{\big\}} \cap A_r\Big].
        \end{aligned}
    \ee
    Inserting \eqref{eq:1st} and \eqref{eq:2nd} in \eqref{unionArprime} gives
    \be
        \Proba_r\big[\fix{\big\{}\tau(2r)<\tau(\Delta)\fix{\big\}} \setminus A'_r\big] \leq c \exp\left(-ar^{\beta}\right)\Proba_r\big[\tau(2r) < \tau(\Delta)\big]
    \ee
  and the lemma follows since $A_r' \subset \fix{\Big\{}\big|\arg(B_{\tau(2r)})-\arg(B_0)\big| \leq r^{-\beta/2}\fix{\Big\}}$.
\end{proof}

This estimate on the “local scale” (from $r$ to $2r$) is strong enough to prove an estimate on the “global scale” (from $r$ to $2^nr$):
\begin{Lem}\label{rtoR}
    For all $\eta, \delta > 0$, there exists $r_0 \geq 1$ such that for all $r \geq r_0$, for all $n \in \N$,
    \begin{equation}
        \Q_r\Big[\tau(2^nr) < \tau(\Delta) , \big|\arg(B_{\tau(2^nr)}) - \arg(B_0)\big| \geq \eta\Big] \leq \delta \Q_x\big[\tau(2^nr) < \tau(\Delta)\big]
    \end{equation}
\end{Lem}
\begin{proof}
    Let $\delta, \eta >0$.
    \fix{Choose} $c, a >0$ as in Lemma~\ref{Arprime}. By applying the strong Markov property $n$ times, for all $r$ large enough such that Lemma \ref{Arprime} applies
    \be
        \begin{aligned}
            &\Proba_r\bigg[\tau(2^nr)<\tau(\Delta) , \bigcap_{i=0}^{n-1} \fix{\big\{}|\arg(B_{\tau(2^{i+1}r)})-\arg(B_{\tau(2^ir)})|\leq (2^ir)^{-\beta/2}\fix{\big\}}\bigg]\\
            &= \prod_{i=0}^{n-1}\Proba_{2^ir}\fix{\big[}\tau(2^{i+1}r)<\tau(\Delta) , |\arg(B_{\tau(2^{i+1}r)})-\arg(B_{\tau(2^ir)})|\leq (2^ir)^{-\beta/2}\fix{\big]}\\
            &\overset{\ref{Arprime}}{\geq} \prod_{i=0}^{n-1}\Proba_{2^ir}[\tau(2^{i+1}r)<\tau(\Delta)]\prod_{i=0}^{n-1}(1-c\exp(-a(2^ir)^{\beta/2})\\
            &= \Proba_r[\tau(2^nr)<\tau(\Delta)] \prod_{i=0}^{n-1}(1-c\exp(-a(2^ir)^{\beta/2}).
        \end{aligned}
    \ee
    It is straightforward to control the product term, using that $\ln(1-u) \geq -2u$ for $u>0$ small enough. For $r$ large enough, for all $n$,
    \begin{equation}
        \begin{aligned}
            \prod_{i=0}^{n-1}\big(1 - c\exp(-a2^{i\beta/2}r^{\beta/2})\big) 
            &= \exp\bigg(\sum_{i=0}^{n-1}\ln\big(1 - c\exp(-a2^{i\beta/2}r^{\beta/2})\big)\bigg)\\
            &\geq \exp\bigg(-2c\sum_{i=0}^{n-1}\exp(-a2^{i\beta/2}r^{\beta/2})\bigg)\\
            &\geq \exp\bigg(-2c\sum_{i=0}^{\infty}\exp(-a2^{i\beta/2}r^{\beta/2})\bigg) \geq 1-\delta
        \end{aligned}
    \end{equation}
    for $r$ large enough since the sum in the exponential is convergent and as small as we want when $r \to \infty$ (it is of the order of its first term). Since on the event $\bigcap_{i=0}^{n-1} \fix{\big\{}|\arg(B_{\tau(2^{i+1}r)})-\arg(B_{\tau(2^ir)})|\leq (2^ir)^{-\beta/2}\fix{\big\}}$,
    \be
        |\arg(B_{\tau(2^nr)}) - \arg(x)\big| \leq \sum_{i=0}^{n-1} |\arg(B_{\tau(2^{i+1}r)})-\arg(B_{\tau(2^ir)})| \leq \sum_{i=0}^{\infty}(2^ir)^{-\beta/2} \leq \eta 
    \ee
    for $r$ large enough, this concludes the proof.
\end{proof}

Using this result, we can prove a final lemma:
\begin{Lem}\label{exhaust}
    Define two exhaustions of $\R^2$: $\Lambda^+_R = B(R) \cup \big(\R_{\fix{+}} \times \R\big)$ and $\Lambda^-_R = B(R) \cup \big(\R_- \times \R\big)$. Then, for all $\delta > 0$, there exists $r_0 \geq 1$ such that for all $r \geq r_0$, for all $n$,
    \begin{equation}
        \Proba_{(0,r)}\big[\tau((\Lambda^+_{2^nr})^c) < \tau(\Delta)\big] \geq (1-\delta)\Proba_r[\tau(2^nr)<\tau(\Delta)]
    \end{equation}
    while
    \begin{equation}
        \Proba_{(0,r)}\big[\tau((\Lambda^-_{2^nr})^c) < \tau(\Delta)\big] \leq \delta \Proba_r[\tau(2^nr)<\tau(\Delta)],
    \end{equation}
    \fix{where we use the notation~\eqref{eq:abuse} since $|(0,r)| = r$.}
\end{Lem}
\begin{proof}
    Let $\delta > 0$, $\eta = \pi/4$, and $r_0 \geq 1$ large enough associated to $\delta/2, \eta$ by Lemma \ref{rtoR}. Observe that for all $r \geq 1$, $n \in \N$, starting from $(r,0)$, we have $\tau(2^nr) \leq \tau\big((\Lambda^+_{2^nr})^c\big)$ and when $\arg\big(B_{\tau(2^nr)}\big) \in [-\pi/4, \pi/4]~mod~2\pi$, $\tau\big((\Lambda^+_{2^nr})^c\big) = \tau(2^nr)$. Hence for $r \geq r_0$,
    \begin{equation}\label{eq:counterex:lb}
        \begin{aligned}
            \Proba_{(r,0)}\big[\tau((\Lambda^+_{2^nr})^c) < \tau(\Delta)\big] 
            &\geq \Proba_{\fix{(r,0)}}\Big[\tau(2^nr) < \tau(\Delta) , \big|\arg(B_{\tau(2^nr)}| \leq \pi/4\big|\Big]\\
            &\geq (1-\delta/2)\Proba_{\fix{(r,0)}}\big[\tau(2^nr) < \tau(\Delta)\big]
        \end{aligned}
    \end{equation}
    by Lemma~\ref{exhaust}. This gives the first statement of the lemma. On the other hand, for all $r \geq r_0$ and $n \in \N$, by a union bound, the strong Markov property with respect to $\tau(2^nr)$ and~\eqref{eq:counterex:lb},
    \begin{equation}\label{eq:inter:lb}
        \begin{aligned}
            \Proba_{(r,0)}\big[\tau\big((\Lambda^-_{2^nr})^c\big) < \tau(\Delta)\big]
            &\leq \Proba_{(r,0)}\big[\tau(2^nr) < \tau(\Delta)\big]\bigg(\sup_{|z| = 2^nr, \arg(z) \in [-\pi/4, \pi/4]} \Proba_z\big[\tau\big((\Lambda^-_{2^nr})^c\big) < \tau(\Delta)\big]\bigg) \\
            &~~~~~~~~~~+ \frac{\delta}{2} \Proba_{(r,0)}\big[\tau(2^nr) < \tau(\Delta)\big].
    \end{aligned}
    \end{equation}
    To conclude, we just have to show that the supremum is smaller than $\delta/2$ for $r$ large enough and all $n \in \N$. This is clear because for any such $z$, the distance between $z$ and $(\Lambda^-_{2^nr})^c$ is at least $2^nr$, so $\tau(B(z,2^nr)^c) \leq \tau((\Lambda^-_{2^nr})^c)$ and
    \be
        \Proba_z\big[\tau\big((\Lambda^-_{2^nr})^c\big) < \tau(\Delta)\big] \leq \Proba_z[\tau(B(z,2^nr)^c) < \tau(\Delta)].
    \ee
    This probability is easily bounded by the same argument as in the proof of Lemma \ref{lbrto2r}: for all $t \geq 0$, by~\eqref{eq:2d:LDP}, if $|z|=2^nr$
    \be
        \Proba_z[\tau(B(z,2^nr)^c) \leq t] \leq 4\exp\left(-\frac{(2^nr)^2}{8t}\right).
    \ee
    On the other hand, on the event $\fix{\{}\tau(B(z,2^nr)^c) \geq t\fix{\}}$ the killing is at least $\K(2^{n+1}r)$ for a time $t$, so
    \be
        \Proba_z\left[\tau(B(z,2^nr)^c) \leq \tau(\Delta) , \tau(B(z,2^nr)^c) \geq t\right] \leq \exp\left(-t(2^{n+1}r)^{-\alpha}\right).
    \ee
    Hence for any such $z$:
    \be
        \Proba_z\big[\tau\big(B(z,2^nr)^c)\big) < \tau(\Delta)\big] \leq 4\exp\left(-\frac{(2^nr)^2}{8t}\right) + \exp\left(-t(2^{n+1}r)^{-\alpha}\right).
    \ee
    Taking for example $t = (2^nr)^{2-2\beta}$ shows that this probability is exponentially small in $2^nr$, in particular it is smaller than $\frac{\delta}{2}$ for $r$ large enough so~\eqref{eq:inter:lb} gives
    \begin{equation}
        \Proba_x\big[\tau((\Lambda^-_{2^nr})^c) < \tau(\Delta)\big] \leq \delta \Proba_x\big[\tau(2^nr) < \tau(\Delta)\big],
    \end{equation}
    which is the second statement of the lemma.
\end{proof}

We are finally ready to prove Theorem \ref{thm.counterex}, which is a direct corollary of the last lemma:
\begin{proof}
    Take $\Lambda^1 = \Lambda^+$, $\Lambda^2 = \Lambda^-$, $r_0$ associated with some fixed $\delta > 0$ by Lemma \ref{exhaust} and $x = (r,0)$, $y = (-r,0)$ for some fixed $r \geq r_0$. By symmetry, for all $n \in \N$,
    \begin{equation}
        \Proba_{(r,0)}\big[\tau\big((\Lambda^+_{2^nr})^c\big) < \tau(\Delta)\big] = \Proba_{(-r,0)}\big[\tau\big((\Lambda^-_{2^nr})^c\big) < \tau(\Delta)\big]
    \end{equation}
    and
    \begin{equation}
        \Proba_{(r,0)}\big[\tau\big((\Lambda^-_{2^nr})^c\big) < \tau(\Delta)\big] = \Proba_{(-r,0)}\big[\tau\big((\Lambda^+_{2^nr})^c\big) < \tau(\Delta)\big].
    \end{equation}
    Then, by Lemma \ref{exhaust}, for all $n \in \N$,
    \begin{equation}
        \frac{\Proba_{(-r,0)}\big[\tau\big((\Lambda^+_{2^nr})^c\big) < \tau(\Delta)\big]}{\Proba_{(r,0)}\big[\tau\big((\Lambda^+_{2^nr})^c\big) < \tau(\Delta)\big]} = \frac{\Proba_{(r,0)}\big[\tau\big((\Lambda^-_{2^nr})^c\big) < \tau(\Delta)\big]}{\Proba_{(r,0)}\big[\tau\big((\Lambda^+_{2^nr})^c\big) < \tau(\Delta)\big]} \leq \frac{\delta}{1-\delta}
    \end{equation}
    and similarly
    \begin{equation}
        \frac{\Proba_{(-r,0)}\big[\tau\big((\Lambda^-_{2^nr})^c\big) < \tau(\Delta)\big]}{\Proba_{(r,0)}\big[\tau\big((\Lambda^-_{2^nr})^c\big) < \tau(\Delta)\big]} \geq \frac{1-\delta}{\delta}.
    \end{equation}
    Hence the limits in the statement of Theorem \ref{thm.counterex} cannot coincide when $\delta$ is small enough ($\delta = 1/4$ for example), which proves the theorem.
\end{proof}


\section{From the killed \fix{B}rownian motion to the killed random walk}\label{appendixA}
In this section, $\alpha \in (14/9,2)$ is fixed, and we let $\K: \R^2 \to \R_+, x \to 1 \wedge \frac{1}{|x|^{\alpha}}$. By restriction, $\K$ also defines a killing function on $\Z^2$. We prove Theorem \ref{krwcounterex}.

\subsection{Notation and strong approximation Theorem}
We start by a warning on notation: in all this section, $a,c$ denote absolute constants (depending only on $\alpha$) that are allowed to vary from line to line.\par

\paragraph{The strong approximation theorem}
If $(S_n)_{n \in \N}$ is a SRW and $(B_t)_{t \in \R_{\fix{+}}}$ a BM defined on the same probability space, let
\be
    \Theta(B,S,n) = \max \{|B_{\lfloor t \rfloor} - S_t|, 0 \leq t \leq n\}.
\ee
We recall the strong approximation of a BM by a SRW as presented for example in Theorem 3.4.2 of \cite{modernRW}.
\begin{Thm}\label{strong}
    There exist two constants $c,a > 0$ and a probability space $(\Omega, \F, \Proba)$ on which are defined a BM $(B_t)_{t \geq 0}$ and a SRW $(S_n)_{n \in \fix{\NN}}$ such that for all $n \geq 1$ and $1 \leq u \leq n^{1/4}$,
    \begin{equation}
        \Proba\Big[\Theta(B,S,n) > u n^{1/4}\sqrt{\log n}\Big] \leq c\exp(-au).
    \end{equation}
\end{Thm}
Let $B$ and $S$ be as in the theorem and denote by $(\F_t)_{t \in \R_{\fix{+}}}$ and $(\G_n)_{n \in \fix{\NN}}$ their respective natural filtrations. We use the same notations as in the preceding sections $\tau(A)$, $\tau(r)$ for the hitting times but we write respectively $\tau^S$ and $\tau^B$ for the hitting time by the SRW and BM.

\paragraph{Coupling the KBM and the SRW}
Let $\xi^B$, $\xi^S$ be two exponential variables on $(\Omega, \F, \Proba)$ with common law $\exp(-\xi)d\xi$ on $\R_+$, independent of each other and of $(B,S)$ (such variables can be found up to enlarging the probability space). The killing times are defined as
\be
    \left\{
            \begin{array}{ll}
                \tau^B(\Delta) = \tau^B(\Delta)(B,\xi^B) = \inf \Bigg\{t \geq 0 \bigg| \int_0^t \K(B_s)ds \geq \xi^B\Bigg\}\\
                \tau^S(\Delta) = \tau^S(\Delta)(S,\xi^S) = \inf \Bigg\{n \geq 0 \bigg| -\sum_{i=0}^{n-1} \ln(1-\K(S_i)) \geq \xi^S\Bigg\}.
            \end{array}
    \right.
\ee
Then, $(B, \tau^B(\Delta))$ satisfies Proposition \ref{prop:kbm} and by independence of $S$ and $\xi^S$ and exponential distribution of $\xi^S$:
\begin{Prop}\label{prop:krw}
    For all $f$ real-valued measurable function on $\R^{\N}$ and $\tau$ stopping time adapted to $(\G_n)$,
    \be
        \E_x\big[f(S)\mathbbm{1}_{\tau \leq \tau(\Delta)}\big] = \E_x\Bigg[f(S)\prod_{i=0}^{\tau-1}(1-\K(S_i))\Bigg].
    \ee
\end{Prop}
This coincides with the usual definition of the killing time of the KRW.
\paragraph{Skew-product decomposition} As in the preceding subsection, we can write
\be
    B_t = \beta_t\exp(i(\arg(x)+W_{R_t}))
\ee
where $\beta$ is a Bessel process of dimension $2$, $R$ is the associated Bessel clock and $W$ is an independent one-dimensional BM started at $0$.
\paragraph{Basic large deviation estimates.} We will need a discrete analog of the basic large deviation estimate~\eqref{eq:2d:LDP}. It is contained in Proposition 2.1.2 of \cite{modernRW}, which we now recall. There exists $a,c >0$ such that for all $x \in \Z^2$, $n \geq 0$, $t >0$,
\be\label{eq:LDP:SRW}
    \Proba_x[\tau^S(B(x,t\sqrt{n})^c) \leq n] \leq c\exp(-at^2)
\ee

\subsection{Auxiliary results and proof of the main theorem}
\paragraph{Overview of the proof}
The proof of Theorem \ref{thm.counterex} follows the same steps as the proof in the preceding section: it studies the event $\fix{\{}\tau(R)<\tau(\Delta)\fix{\}}$ by means of large deviations estimates and crude bounds on the killing function to show that the only way for the walk to escape is to escape fast without approaching the origin where the killing is strong. Many results are direct analogs of continuous results of the preceding section. The additional difficulty lies in the lack of rotational invariance which prevents us from using directly the skew product decomposition for the KRW. To bypass this problem, we use the strong approximation of the SRW by the BM to show that the SRW is not far from \fix{being rotationally invariant}. We must also show that when the trajectories of the BM and the SRW are close, the killing along the trajectory is roughly the same: luckily for us, this is true for the trajectories that escape fast and do not approach the origin.\par

We define, as in the preceding section $\beta > 0$ by $2-\alpha-2\beta=2\beta$. We also let, for all $r > 0$:
\be
    \theta_r = r^{3(1-\beta)/4}\sqrt{\ln(r)}.
\ee
This choice is arbitrary for the moment but will soon become clear. We first give a lemma on the KBM that we will use later in the proof.
\begin{Lem}\label{lem:KBM:theta}
    If $\eps \in (0, \min(\alpha -7/4, 1/8))$, for all $r$ large enough,
    \be
        \Proba_r[\tau^B(r+2\theta_r)< \tau(\Delta)] \geq 1-r^{-\eps}
    \ee
\end{Lem}
\begin{proof}
    By rotational invariance, we can prove it starting from $x = (r,0)$. We first note that if $W^x$ is the horizontal component of the BM and $\sigma^x(r + 2\theta_r)$ is the first hitting time by $W_x$, by reflection principle and scaling for the one dimensional BM, for any $t_r \geq \theta_r$
    \be
        \Proba_r[\sigma^x(r+2\theta_r) \geq t_r^2]
        =\Proba_0\bigg[\sup_{[0,t_r^2]}W^x \leq 2\theta_r\bigg]
        =\Proba_0[|W^x_{t_r^2}|\leq 2\theta_r]
        =\Proba_0\bigg[|W^x_1|\leq \frac{2\theta_r}{t_r}\bigg]
        \leq c\frac{\theta_r}{t_r}.
    \ee
    where $c$ in an absolute positive constant since
    \be
        \frac{1}{\sqrt{2\pi}}\int_{-u}^u \exp\left(-\frac{x^2}{2}\right)dx \leq \frac{2u}{\sqrt{2\pi}} = cu.
    \ee
    Letting $t_r = r^{7/8}$ gives, for $r$ large enough
    \be
        \Proba_r[\sigma^x(r+2\theta_r) \geq t_r^2] \leq cr^{-1/8-3\beta/4}\sqrt{\ln(r)} \leq cr^{-1/8}.
    \ee
    On the other hand, by~\eqref{eq:2d:LDP}
    \be
        \Proba_{x}[\tau^B(B(x,r/2)^c) \leq t_r^2] \leq 4\exp\left(-\frac{r^{1/4}}{32}\right).
    \ee
    If the BM starts from $x = (r,0)$, $(\sigma^x(r+2\theta_r) < t_r^2) \subset \big(\tau^B(r+2\theta_r) < t_r^2\big)$. Besides, on the event $\fix{\{}\tau^B(B(x, r/2)^c) > t_r^2\fix{\}}$, the killing rate is at most (for $r$ large enough) $\K(r/2) \leq 2^{\alpha}r^{-\alpha}$ for a time $t_r^2$ so
    \be
        \begin{aligned}
            \Proba_r[\tau^B(r+2\theta_r)< \tau(\Delta)] 
            &\geq \Proba_r\big[\tau^B(r+2\theta_r)< \tau(\Delta) , \sigma^x(r+2\theta_r) < t_r^2 , \tau^B(B(x, r/2)^c) > t_r^2\big]\\ 
            &\geq \exp(-2^{\alpha}r^{-\alpha}t_r^2)\Proba_r[\sigma^x(r+2\theta_r) < t_r^2 , \tau^B(B(x, r/2)^c) > t_r^2]\\
            &\geq \exp(-2^{-\alpha}r^{7/4-\alpha})\left(1-cr^{-1/8}-4\exp\left(-\frac{r^{1/4}}{32}\right)\right) \geq 1-r^{-\eps}
        \end{aligned}
    \ee
    for any $\eps < \min(a-7/4, 1/8)$ and $r$ large enough.
\end{proof}
We state and prove an analog of Lemma \ref{lbrto2r}, with similar arguments:
\begin{Lem}\label{lem:ub:SRW}
    There exists constants $c,a >0$ such that for all $r >0$ large enough, for all $x \in \partial B(r)$,
    \be
        \Proba_x[\tau^S(2r-\theta_r)< \tau^S(\Delta)] \leq c\exp(-ar^{2\beta}).
    \ee
\end{Lem}
\begin{proof}
    As in the proof of Lemma \ref{lbrto2r}, using the discrete large deviation estimate~\eqref{eq:LDP:SRW} instead of the \fix{continuous} large deviation estimate~\eqref{eq:2d:LDP}, for $r$ large enough:
    \be
        \begin{aligned}
            \Proba_x[\tau^S(2r-\theta_r)< \tau^S(\Delta)]
            &\leq \Proba_x[\tau^S(2r-\theta_r) < r^{2-2\beta}] + \Proba_x[r^{2-2\beta} \leq \tau^S(2r-\theta_r) <\tau(\Delta)]\\
            &\overset{\eqref{eq:LDP:SRW},\ref{prop:krw}}{\leq} c\exp(-ar^{2\beta}) + \E_x\bigg[\mathbbm{1}_{r^{2-2\beta} \leq \tau^S(2r-\theta_r)}\prod_{i=0}^{\tau^S(2r-\theta_r)-1}(1-\K(S_i))\bigg]\\
            &\leq c\exp(-ar^{2\beta}) + \exp\left(r^{2-2\beta}\ln(1-\K(2r-\theta_r))\right)\\
            &\leq c\exp(-ar^{2\beta}) + \exp(-ar^{2-2\beta-\alpha}) \leq c\exp(-ar^{2\beta})
        \end{aligned}
    \ee
\end{proof}
We are now ready to compare the KBM and KRW.
\begin{Lem}\label{lem:lb:SRW}
    For any $0 < \eps < \min((3\alpha-2)/2, (5\alpha-6)/16)$, there exists $c > 0$ such that for all $r$ large enough, for all $x \in \partial B(r)$, 
    \be
        \Proba_x[\tau^S(2r-\theta_r) < \tau^S(\Delta)] \geq (1-cr^{-\eps})\Proba_r[\tau^B(2r) < \tau^B(\Delta)]
    \ee
\end{Lem}
\begin{proof}
    Let $T_0$, $c,c',a,a'$ be absolute constants as in Lemma \ref{rto2r}, and define
    \be
    \left\{
    \begin{array}{ll}
        A_r^B &:= \fix{\{}\tau^B(2r) < \tau^B(r/2)\fix{\}} \cap \fix{\{}\tau^B(2r)<T_0r^{2-2\beta}\fix{\}}\\
        C_r &:= \fix{\{}\Theta(B,S,T_0r^{2-2\beta})\leq \theta_r\fix{\}}
    \end{array}
    \right.
    \ee
    \fix{Theorem~\ref{strong}} with $n = T_0r^{2-2\beta}$ implies
    \be\label{eq:strong:approx}
        \Proba[C_r^c] \leq c\exp\big(-ar^{(1-\beta)/4}\big)
    \ee
      for a constant $a$ which is absolute since it depends only on $\beta, T_0$ which are themselves absolute constants. By Proposition \ref{prop:krw}, for all $x \in \partial B(r)$
    \be
        \begin{aligned}
        \Proba_x[\tau^S(2r-\theta_r) < \tau^S(\Delta)] &\geq \Proba_x\left[\fix{\big\{}\tau^S(2r-\theta_r) < \tau^S(\Delta)\fix{\big\}}\cap C_r \cap A^B_r\right]\\
        &= \E_x\left[\mathbbm{1}_{C_r \cap A^B_r}\prod_{i=0}^{\tau^S(2r-\theta_r)-1}(1-\K(S_i))\right]\\
        &\geq \E_x\left[\mathbbm{1}_{C_r \cap A^B_r}\prod_{i=0}^{\tau^B(2r)-1}(1-\K(S_i))\right]
        \end{aligned}
    \ee
    where the last line is justified by the fact that on the event $C_r \cap A_r^B$, $\tau^S(2r-\theta_r) \leq \tau^B(2r)$: if $B$ leaves the ball $B(2r)$, $S$ (which is close) must have already left a slightly smaller ball $B(2r-\theta_r)$. The interest of this equation is that on the event $C_r \cap A_r^B$, the killing is almost the same for the KBM and KRW. First of all, on the event $C_r \cap A_r^B$, using that $\ln(1-u)\geq -u-u^2$ for $|u|$ small enough:
    \be
        \begin{aligned}
            \prod_{i=0}^{\tau^B(2r)-1}(1-\K(S_i)) &= \exp\left(\sum_{i=0}^{\tau^B(2r)-1}\ln(1-\K(S_i))\right) &\geq \exp\left(-\sum_{i=0}^{\tau^B(2r)-1}(\K(S_i)+\K(S_i)^2)\right).
        \end{aligned}
    \ee
    On $C_r \cap A_r^B$, $\tau^B(r/2) \leq \tau^B(2r)$, and hence for $t \leq \tau^B(2r)$, $|B_t| \geq r/2$ and $|S_t| \geq r/2 - \theta_r \geq r/4$ (for $r$ large enough), and also $\tau^B(2r) \leq T_0r^{2-2\beta}$ so there exists an absolute constant $a$ such that
    \be
        \sum_{i=0}^{\tau^B(2r)-1}\K(S_i)^2 \leq T_0r^{2-2\beta}16^{-\alpha}r^{-2\alpha} \leq ar^{(2-3\alpha)/2}
    \ee
    and another absolute constant $a$ such that
    \be\label{eq:control:killing}
        \begin{aligned}
        \left|\sum_{i=0}^{\tau^B(2r)-1}\K(S_i)-\int_{0}^{\tau^B(2r)}\K(B_s)ds \right|
        &\leq \tau^B(2r)\sup_{0\leq t \leq \tau^B(2r)}|\K(B_t)-\K(S_{\lfloor t \rfloor})|\\
        &\leq Tr^{2-2\beta}\theta_r\sup_{|u|\geq r/4}|\K'(u)| \leq cr^{(6-5\alpha)/16}\sqrt{\ln(r)}.
        \end{aligned}
    \ee
    \fix{Choose any} $0 < \eps < \min((3\alpha-2)/2, (5\alpha-6)/16)$ (this is possible since $\alpha > 14/9 > 6/5$). Putting all these equations together and using that $\exp(1+u) \leq 1+2u$ for $u$ small enough gives, for $r$ large enough
    \be
        \begin{aligned}
            \Proba_x[\tau^S(2r-\theta_r) < \tau^S(\Delta)] 
            &\geq (1-cr^{-\eps})\E_x\left[\mathbbm{1}_{C_r \cap A_r^B} \exp\left(-\int_{0}^{\tau^B(2r)}\K(B_s)ds\right) \right]\\
            &= (1-cr^{-\eps})\Proba_x\big[\fix{\{}\tau^B(2r) < \tau^B(\Delta) \fix{\}} \cap A_r^B \cap C_r\big]\\
            &\geq (1-cr^{-\eps})\Proba_x\big[\fix{\{}\tau^B(2r) < \tau^B(\Delta)\fix{\}}\cap A_r^B\big] -\Proba[C_r^c].
        \end{aligned}
    \ee
    Using~\eqref{eq:strong:approx}, Lemma \ref{rto2r} and the fact that $(1-\beta)/4 > 2\beta$ (which holds since $a > 14/9$) gives
    \be
        \Proba_x[\tau^S(2r-\theta_r) < \tau^S(\Delta)] \geq (1-cr^{\eps})\Proba_r\big[\tau^B(2r) < \tau^B(\Delta)\big]
    \ee
    which is the statement of the lemma.
\end{proof}
We are ready to analy\fix{z}e precisely the event $\fix{\{}\tau^S(2r-\theta_r) < \tau^S(\Delta)\fix{\}}$. For all $r \geq 1$, define similarly to $A^B_r$, for all $L \geq 1$,
\be
    A_r^S = A_r^S(L):= \fix{\{}\tau^S(2r-\theta_r) \leq Lr^{2-2\beta}\fix{\}} \cap \fix{\{}\tau^S(2r-\theta_r) \leq \tau^S((r+\theta_r)/2)\fix{\}}.
\ee
As in the continuum, our goal is to show that $A_r^S$ is the only way for the event $\fix{\{}\tau^S(2r-\theta_r) < \tau^S(\Delta)\fix{\}}$ to happen.
\begin{Lem}\label{lem:ub:KRW}
    There exist $L_0 \geq 1$ and absolute constants $a,c >0$ such that for all $r$ large enough, $x \in \partial B(r)$, denoting $A^S_r = A^S_r(L_0)$,
    \be
        \Proba_x\big[\fix{\{}\tau^S(2r-\theta_r) < \tau^S(\Delta)\fix{\}} \setminus A_r^S\big] \leq c\exp(-ar^{2\beta})\sup_{z \in \partial B(r)}\Proba_z\big[\tau^S(2r-\theta_r) < \tau^S(\Delta)\big] 
    \ee
\end{Lem}
\begin{proof}
    \fix{We adapt the arguments from the proof of Lemma~\ref{rto2r} to the discrete setting}.
    We note that $\fix{\{}\tau^S(2r-\theta_r) < \tau^S(\Delta)\fix{\}} \setminus A^S_r = \fix{\big\{}\tau^S(2r-\theta_r) < \tau^S(\Delta)\fix{\big\}} \cap \fix{\{}E_r^S \cup F^S_r\fix{\}}$
    with 
    \be
        \left\{
            \begin{array}{ll}
                E_r^S = \fix{\big\{}\tau^S((r+\theta_r)/2)<\tau^S(2r-\theta_r)\fix{\big\}}\\
                F_r^S = \fix{\big\{}\tau^S(2r-\theta_r) > Lr^{2-2\beta}\fix{\big\}}.
            \end{array}
        \right.
    \ee
    On the one hand, by the strong Markov property and Lemma \ref{lem:ub:SRW}
    \be
        \begin{aligned}
            \Proba_x\big[\fix{\{}\tau^S(2r-\theta_r) &< \tau^S(\Delta)\fix{\}} \cap E_r^S\big]\\
            &\leq \sup_{y \in \partial B((r+\theta_r)/2)}\left\{\Proba_{y}[\tau^S(r) < \tau^S(\Delta)]\right\} \sup_{z \in \partial B(r)}\left\{\Proba_z[\tau^S(2r-\theta_r) < \tau^S(\Delta)]\right\}\\
            &\leq c\exp(-ar^{2\beta})\sup_{z \in \partial B(r)}\left\{\Proba_z[\tau^S(2r-\theta_r) < \tau^S(\Delta)]\right\}
        \end{aligned}
    \ee
    On the other hand, since on the event $\fix{\{}\tau^S(2r-\theta_r) < \tau^S(\Delta)\fix{\}} \cap F^S_r$ the killing is at least $\K(2r)$ for a time at least $Lr^{2-2\beta}$,
    \be
        \begin{aligned}
            \Proba_x\big[\fix{\{}\tau^S(2r-\theta_r) < \tau^S(\Delta)\fix{\}} \cap F_r^S\big] 
            \leq \exp(-cLr^{2\beta}) 
            &\overset{\eqref{eq:lower:bound:KBM}}{\leq} c\exp(-ar^{2\beta}) \Proba_r\big[\tau^B(2r)< \tau^B(\Delta)\big]\\
            &\overset{\ref{lem:lb:SRW}}{\leq} c\exp(-ar^{2\beta})\Proba_x[\tau^S(2r-\theta_r) < \tau^S(\Delta)]
        \end{aligned}
    \ee
    for some fixed $L=L_0$ large enough and any $x \in \partial B(r)$. Hence,
    \be
        \Proba_x\big[\fix{\{}\tau^S(2r-\theta_r) < \tau^S(\Delta)\fix{\}} \setminus A_r^S\big] \leq c\exp(-ar^{2\beta})\sup_{z \in \partial B(r)}\Proba_x\big[\tau^S(2r-\theta_r) < \tau^S(\Delta)\big].
    \ee
\end{proof}
For future use, we also need to compare the event $\fix{\{}\tau^S(2r-\theta_r) < \tau^S(\Delta)\fix{\}}$ \fix{to its continuous} analog.
\begin{Lem}\label{lem:ub:KRW/KBM}
    There exist $L_0 \geq 1$ (the same as in the preceding lemma) and $c, \eps >0$ such that for all $r$, $x \in \partial B(r)$
    \be
        \Proba_x\big[\fix{\{}\tau^S(2r-\theta_r) < \tau^S(\Delta)\fix{\}} \cap A_r^S\big] \leq (1+cr^{-\eps})\Proba_r\big[\tau^B(2r) < \tau^B(\Delta)\big]
    \ee
\end{Lem}
\begin{proof}
    We use the same argument as in the proof of Lemma \ref{lem:lb:SRW}, using that with high probability the SRW and KRW are close by the strong approximation theorem and that on $A_r^S \cap C_r$ the killing is almost the same for the KBM and the KRW. More precisely, for any $z \in \partial B(r)$, by~\eqref{eq:strong:approx} and Proposition \ref{prop:krw}
    \be
        \begin{aligned}
        \Proba_z[\fix{\{}\tau^S(2r-\theta_r) < \tau^S(\Delta)\fix{\}} \cap A_r^S] 
        &\leq \Proba_z[\fix{\{}\tau^S(2r-\theta_r) < \tau^S(\Delta)\fix{\}} \cap A_r^S \cap C_r] + \Proba_z[C_r^c]\\
        &\leq \E_z\left[\mathbbm{1}_{A_r^S \cap C_r}\prod_{i=0}^{\tau^S(2r-\theta_r)-1}(1-\K(S_i))\right] + c\exp(-ar^{(1-\beta)/4}).
        \end{aligned}
    \ee
    \fix{We proceed as in Lemma~\ref{lem:lb:SRW}, but we bound from above instead of bounding from below:} using that on $C_r$, $\tau^B(2r-2\theta_r) \leq \tau^S(2r-\theta_r)$ and that $\ln(1-u)\leq u$, on the event $C_r$:
    \be
        \prod_{i=0}^{\tau^S(2r-\theta_r)-1}(1-\K(S_i)) \leq \prod_{i=0}^{\tau^B(2r-2\theta_r)-1}(1-\K(S_i)) \leq \exp\left(-\sum_{i=0}^{\tau^B(2r-2\theta_r)}\K\big(S_i\big)\right).
    \ee
    Similarly to~\eqref{eq:control:killing}, on $A_r^S \cap C_r$ we have for all $t \leq \tau^B(2r-2\theta_r), |S_{\lfloor t \rfloor}| \geq r/2$ so $|B_t|\geq |S_{\lfloor t \rfloor}| - \theta_r \geq r/4$ (for $r$ large enough). Hence on the event $A^S_r \cap C_r$:
    \be
        \begin{aligned}
            \left|\sum_{i=0}^{\tau^B(2r-2\theta_r)-1}\K(S_i)-\int_{0}^{\tau^B(2r-2\theta_r)}\K(B_s)ds \right|
            &\leq \tau^B(2r-2\theta_r)\sup_{0\leq t \leq \tau^B(2r-2\theta_r)}|\K(B_t)-\K(S_{\lfloor t \rfloor})|\\
            &\leq L_0r^{2-2\beta}\theta_r\sup_{|u|\geq r/4}|\K'(u)| \leq cr^{(6-5\alpha)/16}\sqrt{\ln(r)} < r^{-\eps}.
        \end{aligned}
    \ee
    For any $0<\eps<(5\alpha-6)/16$ and $r$ large enough. Putting everything together, using that $\exp(1+u) \leq 2u$ for $u$ small enough, we get that for $r$ large enough and $z \in \partial B(r)$,
    \be
        \begin{aligned}
            \Proba_z[\fix{\{}\tau^S(2r-\theta_r) < \tau^S(\Delta)\fix{\}} \cap A_r^S] 
            &\leq \E_x\left[\mathbbm{1}_{A_r^S \cap C_r}\exp(r^{-\eps})\exp\left(-\int_0^{\tau^B(2r-2\theta_r)}\K(B_s)ds\right)\right]\\
            &~~~~~~~~+ c \exp(-ar^{(1-\beta)/4})\\
            &\leq (1+2r^{-\eps})\Proba_r[\tau^B(2r-2\theta_r) < \tau^B(\Delta)] + c \exp(-ar^{(1-\beta)/4}).
            \end{aligned}
    \ee
    But Lemma \ref{lem:KBM:theta} and the strong Markov property imply for all $0 < \eps < \alpha -7/4$ and $r$ large enough
    \be\label{eq:r:theta}
        \begin{aligned}
            \Proba_r[\tau^B(2r)<\tau^B(\Delta)] 
            &\geq \Proba_r[\tau^B(2r-2\theta_r)<\tau^B(\Delta)]\Proba_{2r-\theta_r}[\tau^B(2r)<\tau^B(\Delta)]\\
            &\geq (1-r^{-\eps})\Proba_r[\tau^B(2r-2\theta_r)<\tau^B(\Delta)].
        \end{aligned}
    \ee
    The last two equations together imply, for any $0<\eps < \min((6-5\alpha)/16,\alpha-7/4)$, some constant $c>0$ , for all $r$ large enough and $z \in \partial B(r)$
    \be
        \Proba_z[\fix{\{}\tau^S(2r-\theta_r) < \tau^S(\Delta)\fix{\}} \cap A_r^S] \leq (1+cr^{-\eps})\Proba_r[\tau^B(2r)<\tau^B(\Delta)] + c \exp(-ar^{(1-\beta)/4}).
    \ee
    Since $(1-\beta)/4 > 2\beta$, Lemma \ref{lbrto2r} imply for some constant $c>0$
    \be
       \Proba_z[\fix{\{}\tau^S(2r-\theta_r) < \tau^S(\Delta)\fix{\}} \cap A_r^S] 
       \leq (1+cr^{-\eps})\Proba_r[\tau^B(2r)<\tau^B(\Delta)].
    \ee
\end{proof}
We put all these lemmas together to prove a key result:
\begin{Lem}\label{lem:ub}
    There exists $L_0 \geq 1$ (the same as in the preceding lemma), $\eps >0$, $c,a>0$ such that for all $r$ large enough and all $x \in \partial B(r)$,
    \be
        \Proba_x[\tau^S(2r-\theta_r)<\tau^S(\Delta)] \leq (1+cr^{-\eps}) \Proba_r[\tau^B(2r)<\tau^B(\Delta)]
    \ee
\end{Lem}
\begin{proof}
    Let $L_0$ be as in the preceding lemma, $\eps \in (0, \min(\alpha -7/4, (5\alpha-6)/16)$. Let $r$ large enough and $x^{\star} \in \partial B(r)$ maximi\fix{z}ing $\Proba_x[\tau^S(2r-\theta_r)< \tau^S(\Delta)]$ \fix{over the discrete set $\partial B(r)$}. Then, Lemma \ref{lem:ub:KRW} implies 
    \be\label{eq:reverse}
        \Proba_{x^{\star}}[\fix{\{}\tau^S(2r-\theta_r)<\tau^S(\Delta)\fix{\}}\cap A_r^S] \geq (1-c\exp(ar^{-2\beta}))\Proba_{x^{\star}}[\tau^S(2r-\theta_r)<\tau^S(\Delta)]
    \ee
    so
    \be
        \begin{aligned}
            \Proba_r[\tau^B(2r)<\tau^B(\Delta)]  
            &\overset{\ref{lem:ub:KRW/KBM}}{\geq} (1-cr^{-\eps})\Proba_{x^{\star}}[\fix{\{}\tau^S(2r-\theta_r)<\tau^S(\Delta)\fix{\}} \cap A_r^S]\\
            &\overset{\eqref{eq:reverse}}{\geq} (1-c'r^{-\eps})\Proba_{x^{\star}}[\tau^S(2r-\theta_r)<\tau^S(\Delta)].
        \end{aligned}
    \ee
\end{proof}

We now prove the analog of Lemma \ref{Arprime}, i.e. we prove that the KRW does not wind around on the event $(\tau^S(2r)< \tau^S(\Delta)$ (since it stays close to the KBM which does not wind around).

\begin{Lem}\label{lem:KBM:noturn}
    There exist $c,a>0$ such that for all $r$ large enough, $x \in \partial B(r)$,
    \be
        \begin{aligned}
        \Proba_x\bigg[\tau^S&(2r-\theta_r)<\tau^S(\Delta) , \big|\arg\big(S_{\tau^S(2r-\theta_r)}\big)-\arg(S_0)\big|\geq r^{-\beta/2} + \frac{\theta_r}{r}\bigg]\\
        &\leq c\exp(-ar^{\beta})\Proba_x[\tau^S(2r-\theta_r)<\tau^S(\Delta)].
        \end{aligned}
    \ee
\end{Lem}  
\begin{proof}
    \fix{Choose} $L_0, \eps >0$ as in the preceding lemma. As in the proof of Lemma \ref{Arprime}, by union bound
    \be\label{eq:turn:dec}
        \begin{aligned}
            \Proba_x\big[\tau^S(2r-\theta_r)&<\tau^S(\Delta) , |\arg(S_{\tau^S(2r-\theta_r)}-\arg(S_0)|\geq 2r^{-\beta/2}\big]\\
            &\leq \Proba_x[\fix{\{}\tau^S(2r-\theta_r)<\tau^S(\Delta)\fix{\}} \setminus A_r^S] + \Proba_x[C_r^c] + \Proba_x\big[\fix{\big\{}\tau^S(2r-\theta_r)<\tau^S(\Delta)\fix{\big\}} \cap D_r\big]
        \end{aligned}
    \ee
    with
    \be
        D_r = A_r^S \cap C_r \cap \fix{\bigg\{}|\arg(S_{\tau^S(2r-\theta_r)}-\arg(S_0)|> r^{-\beta/2}+\frac{\theta_r}{r}\fix{\bigg\}}.
    \ee
    The first term on the RHS of~\eqref{eq:turn:dec} is bounded by Lemma~\ref{lem:ub:KRW} and Lemma~\ref{lem:ub}:
    \be\label{eq:bound:A}
       \Proba_x\big[\fix{\big\{}\tau^S(2r-\theta_r)<\tau^S(\Delta)\fix{\big\}} \setminus A_r^S\big] \leq c\exp(-ar^{2\beta})\Proba_r\big[\tau^B(2r)<\tau^B(\Delta)\big]. 
    \ee
    and the second one by~\eqref{eq:strong:approx}. To control the third term, we use the computations of the proof of Lemma~\ref{lem:ub:KRW/KBM} (we do not repeat them): on the event $A^S_r \cap C_r$, the killing is almost the same for the KRW and the BRW, so
    \be\label{eq:kbm:B}
        \Proba_x[\fix{\{}\tau^S(2r-\theta_r)<\tau^S(\Delta)\fix{\}} \cap D_r] \leq (1+cr^{-\eps})\Proba_r\big[\fix{\big\{}\tau^B(2r-2\theta_r)<\tau^B(\Delta)\fix{\big\}}\cap D_r\big].
    \ee
    Besides, on the event $D_r$, we have 
    \be
        \big|\arg\big(B_{\tau^S(2r-\theta_r)}\big)-\arg\big(S_{\tau^S(2r-\theta_r)}\big)\big| \leq \frac{\theta_r}{r}
    \ee
    since $|S_{\tau^S(2r-\theta_r)}|= 2r-\theta_r \geq r$ (for $r$ large enough) and $B_{\tau^S(2r-\theta_r)} \in B\big(S_{\tau^S(2r-\theta_r)}, \theta_r\big)$. Hence,
    \be\label{eq:arg:KBM}
        \begin{aligned}
            |\arg(B_{\tau^S(2r-\theta_r)})-\arg(x)| &\geq |\arg(S_{\tau^S(2r-\theta_r)})-\arg(x)| - |\arg(B_{\tau^S(2r-\theta_r)})-\arg(S_{\tau^S(2r-\theta_r)})|\\
            &> r^{-\beta/2}.
        \end{aligned}
    \ee
    Also, on $D_r$, since for all $t \leq \tau^S(2r-\theta_r)$ and $r$ large enough, 
    \be
        |B_t| \geq \big|S_{\lfloor t \rfloor}\big| -\theta_r \geq (r+\theta_r)/2 \geq r/4
    \ee
    and since $\tau^S(2r-\theta_r) \leq L_0r^{2-2\beta}$,
    \be
        R_{\tau^S(2r-\theta_r)} = \int_0^{L_0r^{2-2\beta}}\frac{ds}{|B_s|^2} \leq 16L_0r^{-2\beta}.
    \ee
    Together with \eqref{eq:arg:KBM} and the skew product decomposition, this implies that on $D_r$, 
    \be\label{eq:KBM:sup}
        \sup_{[0,16L_0r^{-2\beta}]}|W_t| \geq r^{-\beta/2}. 
    \ee
    As in the proof of Lemma \ref{Arprime}, since the radial and tangential components of the skew-product decomposition of $B$ are independent and since $\fix{\big\{}\tau^B(2r-2\theta_r)<\tau^B(\Delta)\fix{\big\}}$ involves only the radial component,
    \be\label{eq:tauB}
    	\begin{aligned}
    	\Proba_x\big[\fix{\{}\tau^B(2r-2\theta_r)<\tau^B(\Delta)\fix{\}} \cap D_r\big]
	&\leq \Proba_x\bigg[\tau^B(2r-2\theta_r)<\tau^B(\Delta) , \sup_{[0,16L_0r^{-2\beta}]}|W_t| \geq r^{-\beta/2}\bigg]\\
	&\overset{\eqref{eq:LDP:twosided}}{\leq}c\exp(-ar^{\beta})\Proba_r\big[\tau^B(2r-2\theta_r)<\tau^B(\Delta)\big].
	\end{aligned}
    \ee
    Putting everything together, we obtain
        \be\label{eq:bound:D}
        \begin{aligned}
            \Proba_x\big[\fix{\big\{}\tau^S(2r-\theta_r)<\tau^S(\Delta)\fix{\big\}}\cap D_r\big] 
            &\overset{\eqref{eq:kbm:B}, \eqref{eq:tauB}}{\leq} c\exp(-ar^{\beta})\Proba_r[\tau^B(2r-2\theta_r)<\tau^B(\Delta)] \\
            &\overset{\eqref{eq:r:theta}}{\leq} c'\exp(-ar^{\beta})\Proba_r[\tau^B(2r)<\tau^B(\Delta)]
        \end{aligned}
    \ee
    Inserting the bounds of~\eqref{eq:bound:A},~\eqref{eq:strong:approx} and~\eqref{eq:bound:D} in~\eqref{eq:turn:dec} gives
    \be
    	\begin{aligned}
    	\Proba_x\big[\tau^S(2r-\theta_r)&<\tau^S(\Delta) , |\arg(S_{\tau^S(2r-\theta_r)}-\arg(S_0)|\geq 2r^{-\beta/2}\big]\\
	&\leq c\exp(-ar^{\beta})\Proba_r[\tau^B(2r)<\tau^B(\Delta)] + c\exp(-ar^{(1-\beta)/4})\\
	&\leq c'\exp(-ar^{\beta})\Proba_r[\tau^B(2r)<\tau^B(\Delta)]
	\end{aligned}
   \ee
   by the lower bound of Lemma \ref{rto2r} since $(1-\beta)/4 > 2\beta$. We conclude using Lemma \ref{lem:lb:SRW}.
   \end{proof}

From now on, the proof is strictly the same as the one for the KBM. For any $r >0$, define the sequence $(r_n)_{n \geq 0}$ by $r_0 =r$ and $r_{n+1}=2r_n-\theta_{r_n}$. Then, the following holds:
\begin{Lem}
    For any $\delta, \eta >0$, there exists $\rho > 0$ such that for all $r \geq \rho$, for all $x \in \partial B(r)$, for all $n \geq 0$,
    \be
        \Proba_x\left[\tau^S(r_n)<\tau^S(\Delta) , |\arg(S_{\tau(r_n))-\arg(x)}|\geq \eta\right] \leq \delta \Proba_x[\tau^S(r_n)<\tau^S(\Delta)].
    \ee
\end{Lem}
\begin{proof}
    \fix{We adapt the proof of the last section to the discrete setting.}
    By applying the strong Markov property $n$ times and Lemma \ref{lem:KBM:noturn}, for all $r$ large enough, $x \in \partial B(r)$ and $n \geq 0$ 
    \be
        \begin{aligned}
            \Proba_x\bigg[\fix{\big\{}\tau^S(r_n)&<\tau^S(\Delta)\fix{\big\}} \cap \bigcap_{i=0}^{n-1}\fix{\bigg\{}\big|\arg\big(S_{\tau^S(r_{i+1})}\big)-\arg\big(S_{\tau^S(r_{i+1})}\big)\big|\leq r_i^{-\beta/2}+\frac{\theta_{r_i}}{r_i}\fix{\bigg\}}\bigg]\\
            &\geq \Proba_x\big[\tau^S(r_n)<\tau^S(\Delta)\big]\prod_{i=0}^{n-1}\big(1-c\exp\big(-ar_i^{\beta}\big)\big).
        \end{aligned}
    \ee
    For $r$ large enough, we can use the crude bounds $r_i \geq (3/2)^ir$ and $\frac{\theta_{r_i}}{r_i} \leq r_i^{-1/4} \leq \left(\frac{2}{3}\right)^{i/4}$ for all $i \geq 0$. This implies \fix{on one hand} that
    \be
        \prod_{i=0}^{n-1}\big(1-c\exp\big(-ar_i^{\beta}\big)\big) \geq 1-\delta
    \ee
    for $r$ large enough, and on the other hand that on the event in the right-hand side, for $r$ large enough
    \be 
        \big|\arg\big(S_{\tau^S(r_n)}\big)-\arg(x)\big| \leq \sum_{i=0}^{\infty} \left(r_i^{-\beta/2}+\frac{\theta_{r_i}}{r_i}\right) \leq \eta
    \ee
    which concludes the proof.
\end{proof}
Using this lemma, we can prove the equivalent of Lemma \ref{exhaust}.
\begin{Lem}\label{exhaust:KBM}
    Define two exhaustions of $\Z^2$: $\Lambda^+_R = B(R) \cup \big(\R_{\fix{-}} \times \R\big)$ and $\Lambda^-_R = B(R) \cup \big(\R_{\fix{+}} \times \R\big)$. Then, for all $\delta > 0$, there exists $r_0 \geq 1$ such that for all $r \in \Z, r \geq r_0$, for all $n$,
    \begin{equation}
        \Proba_{(0,r)}\big[\tau((\Lambda^+_{r_n})^c) < \tau(\Delta)\big] \geq (1-\delta)\Proba_r[\tau(r_n)<\tau(\Delta)]
    \end{equation}
    while
    \begin{equation}
        \Proba_{(0,r)}\big[\tau((\Lambda^-_{r_n})^c) < \tau(\Delta)\big] \leq \delta \Proba_r[\tau(r_n)<\tau(\Delta)]
    \end{equation}
\end{Lem}
\begin{proof}
    The proof is exactly the same as the proof of Lemma \ref{exhaust} upon replacing $2^nr$ by $r_n$ and the continuous large deviation estimate~\eqref{eq:2d:LDP} by the discrete large deviation estimate~\eqref{eq:LDP:SRW}.
\end{proof}
Lemma~\ref{exhaust:KBM} finally implies Theorem~\ref{thm.counterex} by the same proof as in the preceding section.

\begin{Rem}\label{rem:exhaustion:finite}
    It is easy to obtain the same result for a finite exhaustion: for $R, \RR >0$, let $\Lambda^{\pm}_{R,\RR} := \Lambda^{\pm}_R \cap B(\RR)$. For all $R >0$ fixed, 
    \be
        \Proba_{(0,r)}[\tau((\Lambda^{\pm}_{R,\RR})^c) < \tau(\Delta)] \overset{\RR \to \infty}{\longrightarrow} \Proba_{(0,r)}[\tau((\Lambda^{\pm}_R)^c) < \tau(\Delta)].
    \ee
    Hence, if for all $R >0$ we define $\RR(R)$ large enough, the limit in distribution along the exhaustions $(\Lambda^{\pm}_{R, \RR(R)})_{R \in \N}$ do not coincide and these are finite subsets of $\Z^d$.
\end{Rem}


\appendix
\section{The random walk conditioned to survive in dimension $2$}
\label{proof2d}
In this section, we adapt the proof of Section \ref{proof3d} to the case $d=2$.

\subsection{Preliminaries}
\paragraph{General notations} In this section, we will write $O(1)$ for a quantity bounded by a universal constant, independ\fix{e}nt of all the parameters. We fix $(\Lambda_R)_{R \in \N}$ an exhaustion of $\Z^2$.

\paragraph{Notations for the SRW and KRW} Recall the notations of Section \ref{proof3d} for paths and hitting times of discrete sets by the SRW. \fix{For $B \subset \Z^2$ a strict subset and $u,v \in \Z^2$, we denote by $\Gamma_B(u,v)$ the set of finite paths starting at $u$, ending at $v$ and such that all vertices of $\gamma$ belong to $B$ except maybe the first and last one. In particular, $\gamma \in \Gamma_{\Z^*}(u,v)$ is any path from $u$ to $v$ avoiding the origin.}
Let $a$ denote the potential kernel of the SRW in $\Z^2$. If $\gamma$ is a path in $\Z^2$, we will again denote by $\s(\gamma)$ its probability weight in the SRW and $\B(\gamma)$ its probability weight in the KRW, i.e. $\B(\gamma) = \s(\gamma) \prod_{i = 0}^{|\gamma|}\big(1-\K(\gamma (i))\big)$. With these notations, for all $x \in \Z^2$
\begin{equation}
  a(x) = a(0,x) = \sum_{\fix{\gamma \in \Gamma_{\Z^*}(0,x)}} \s(\gamma) = \sum_{n > 0} \Proba_0[\tau^+(0)>n, S_n=x]
\end{equation}
is the potential kernel. It is symmetric by rotation by an angle $\pi/2$ since $\Z^2$ is. \fix{Using the symmetry, the Markov property and the fact that 
$
    \sum_{n > 0} \Proba_0[\tau^+(0)\geq n, S_n=0] = \Proba_0[\tau^+(0) < \infty] = 1,
$
we obtain that if $e_i$ is any of the four neighbors of the origin, $a(e_i)=1$}. 

More generally, for $x, y \in \Z^2$,
\begin{equation}
  a(x,y) = \sum_{\fix{\gamma \in \Gamma_{\Z^*}(x,y)}} \s(\gamma) = \sum_{n > 0} \Proba_x[\tau^+(0)>n, S_n=y].
\end{equation}
We define the harmonic measure associated to the KRW:  
\fix{for $u,v \in \Z^2, B \subset \Z^2$,}
\be
    H_B^{\K}(u, v) = \sum_{\fix{\gamma \in \Gamma_B(u,v)}} \B(\gamma).
\ee
This sum converges because since $\Z^2 \setminus B \neq \emptyset$, we can fix $x_0 \in \Z^2 \setminus B$ and all the paths in the sum must avoid $x_0$ so upon translating the sum is dominated by the potential kernel. We will often use the following compact notation: for $u \in \Z^2$, $C \subset \Z^2$, $H_B^{\K}(u, C) = \sum_{v \in C}H_B^{\K}(u, v)$\par
We provide a two-dimensional analog of~\eqref{decomposition}. 
\fix{For $B, C \subset \Z^2$ and $u \in \Z^2$, denote $\Gamma_B(u,C) := \cup_{v \in C} \Gamma_B(u,v)$}.
For all $x \in \Z^2$ and $r_0<r_1< r_2<r_3 \in \R_+$, $R$ large enough so that $B(r_3) \subset \Lambda_R$, we will break the probabilities in several terms, according to: $T_0$ the last time of exit of $B(r_0)$ before $\tau(r_3)$, the first time $T_1 \geq T_0$ when the walk leaves $B(r_1)$, the last time $T_2$ before $T_3 = \tau(r_3)$ when the walk exits $B(r_2)$, and finally the moment when the walk leaves $\Lambda_R$ for the first time:

\begin{equation}
  \label{decomposition2d}
  \begin{aligned}
    \Proba_x[\tau(\Lambda_R^c) < \tau(\Delta)]
    &= \sum_{\fix{\gamma \in \Gamma_{\Lambda_R}(x,\Lambda_R^c)}} \B(\gamma)\\
    &= \sum_{u \in \partial B(r_0), v \in \partial B(r_1), w \in \partial B(r_2), z \in \partial B(r_3)} H^{\K}_{B(r_3)}(x, u)  H^{\K}_{B(r_1) \setminus B(r_0)}(u,v)\\
    &~~~~~\times H^{\K}_{B(r_3) \setminus B(r_0)}(v,w) H^{\K}_{B(r_3) \setminus B(r_2)}(w,z)H^{\K}_{\Lambda_R}(z, \Lambda_R^c).
  \end{aligned}
\end{equation}
This decomposition is represented on the right of Figure \ref{fig:split}.

\subsection{Proof of the Theorem}
\paragraph{Overview of the proof} The proof is almost the same as in the higher dimensional case, but we must add a cutoff for the harmonic measure to be well-defined: we will overlook all paths which exit a big ball $B(r_3)$ where $r_3$ is a well-calibrated parameter. As in the higher dimensional case, we will show that for $v \in \partial B(r_1)$, $w \in \partial B(r_2)$, for well calibrated $r_0, r_1, r_2$ and $r_3$, $H^{\K}_{B(r_3) \setminus B(r_0)}(v,w) \approx H_{B(r_3) \setminus B(r_0)}(v,w)$ depends only on the $r_i$ and not on the exact location of $v$ and $w$, so \eqref{decomposition2d} implies the Theorem in $d=2$ as \eqref{decomposition} implies the Theorem in $d=3$. We state some standard results and a series of lemmas increasing in technical difficulty which will \fix{have an analogous role as Lemma \ref{lemme3.1} in the current proof.}

\begin{Lem}
  \label{a.estimate}
  There exists a constant $k \in \R$ such that the potential kernel satisfies
  \begin{equation}
    a(x) = \frac{2}{\pi} \ln|x| + k + o_{|x| \to \infty}(1).
  \end{equation}
\end{Lem}
This classical estimate can be found for example in \cite{Lawler} (Theorem 1.6.2).

\begin{Lem}
  \label{lem.a(x,y)}
  For all $x, y \in \Z^2 \setminus \{0\}$,
  \begin{equation}
    a(x,y) = \sum_{\fix{\gamma \in \Gamma_{\Z^*}(x,y)}} \s(\gamma) = \frac{\Proba_x[\tau(y) < \tau(0)]}{\Proba_0[\tau(y) < \tau^+(0)]}a(y).
  \end{equation}
\end{Lem}
This lemma is easily proved by splitting the sum according to the first time the paths hit $y$.

\begin{Lem}
  \label{lem.hitting.prob}
  For $x, y \in \Z^2 \setminus \{ 0 \}$,
  \begin{equation}
    \Proba_x[\tau(y) < \tau(0)] = \frac{a(x) + a(y) - a(x-y)}{2a(y)}.
  \end{equation}
  For $y \in \Z^2 \setminus \{ 0 \}$,
  \begin{equation}
    \Proba_0[\tau(y) < \tau^+(0)] = \frac{\fix{1}}{2a(y)}.
  \end{equation}
\end{Lem}
\begin{proof}
    This is a direct consequence of item (v) of Proposition 1.1 of \cite{Gantert}. Indeed, their $\widehat{S}$ is the Doob transform of the SRW by the potential kernel, so for $x,y \in \Z^2 \setminus \{0\}$,
    \be
        \Proba_x[\tau(y) < \tau(0)] 
        = \frac{a(x)}{a(y)}\Proba_x[\widehat{\tau_0}(y) < \infty]
        = \frac{a(x) + a(y) - a(x-y)}{2a(y)}.
    \ee
    The second equality is a special case: for $y \in \Z^2 \setminus \{0\}$, if $(e_i)_{i \in \{1,\dots,4\}}$ denote the four neighbors of the origin:
    \be
        \Proba_0[\tau(y) < \tau^+(0)] = \frac{1}{4}\sum_{i \in \{1,\dots,4\}}\Proba_{e_i}[\tau(y) < \tau(0)] = \frac{1}{4}\sum_{i \in \{1,\dots,4\}} \frac{a(e_i) + a(y) - a(e_i-y)}{2a(y)}.
    \ee
    By symmetry and harmonicity of the potential kernel, 
    \be
        \sum_{i \in \{1,\dots,4\}}(a(y)-a(e_i-y)) = \sum_{i \in \{1,\dots,4\}}(a(y)-a(y-e_i)) = 0.
    \ee
    This gives the second item of the lemma \fix{since, as we already mentioned, $a(e_i)=1$} for all $i$.
\end{proof}

\begin{Lem}
  \label{taur3}
  There exist universal constants $c, C$ such that if $S$ is the SRW in $\Z^2$, for all $r \in \R_+$, $n \in \N$, $x \in B(r)$
  \begin{equation}
    \Proba_x[\tau(r) \geq n] \leq C\exp\bigg(-c\frac{n}{r^2}\bigg).
  \end{equation}
\end{Lem}
This is folklore, but we provide a proof for completeness.
\begin{proof}
    By the central limit theorem, for the SRW started at $0$:
    \be
        \frac{S_{r^2}}{r}\overset{r \to \infty}{\longrightarrow} \mathcal{N}(0,1)
    \ee
    in distribution so there exists $c >0$ such that for all $r \in \N$ large enough $\Proba_0\big[S_{r^2} \geq 2r\big] \geq c$. Hence, for all $r$ large enough, for all $x \in \partial B(r)$,
    \be
        \Proba_x[\tau(r) \geq r^2] \leq 1-c.
    \ee
    For all $i \in \N$, by applying the \fix{s}trong Markov property and this bound we get
    \be
        \Proba_x[\tau(r) \geq ir^2] \leq (1-c)^i = \exp(i \ln(1-c)).
    \ee
    Which gives the lemma.
\end{proof}
We also state a direct consequence of the local central limit theorem (LCLT) in dimension $2$ (see for example Theorem 1.2.1 of \cite{Lawler}):
\begin{Lem}
  \label{LCLT}
  There exists a universal constant $C$ such that for all $x, y \in \Z^2$, $\Proba_x[S_n = y] \leq \frac{C}{n}$.
\end{Lem}

We now state a series of technical lemmas which are more specific to our problem.

\begin{Lem}\label{Hr0}
    For all $r_0$ and $x, y \in B(r_0)^c$, if $2|x| \leq |y|$,
    \begin{equation}
        H_{B(r_0)^c}(x,y) = \frac{2a(y)}{\pi}\Bigg(\ln\frac{|x|}{r_0} + O(1)\Bigg)
    \end{equation}
    While if $2|y| \leq |x|$,
    \begin{equation}
        H_{B(r_0)^c}(x,y) = \frac{2a(y)}{\pi}\Bigg(\ln\frac{|y|}{r_0} + O(1)\Bigg).
    \end{equation}
\end{Lem}
 \fix{We emphasize that,} as announced in the \fix{preliminaries}, the $O(1)$ is a term bounded by a universal constant, in this lemma and everywhere in this section.
\begin{proof}
    \fix{For $B,C \subset \Z^2$, $u,v \in \Z^2$, let $\Gamma_B^C(u,v)$ denote the paths in $\Gamma_B$ that have at least one vertex in $C$, that is the paths from $u$ to $v$ remaining in $B$ (except maybe from the endpoints) and touching $C$.}
    Let $r_0 > 0$ and $x,y \in B(r_0)^c$ with $|x| \leq |y|$. Since a path avoiding $B(r_0)$ also avoids $0$,
    \begin{equation}
        \begin{aligned}
            H_{B(r_0)^c}(x,y) 
            &= \sum_{\gamma \in \Gamma_{B(r_0)^c}(x, y)}\s(\gamma)\\
            &= \sum_{\gamma \in \Gamma_{\Z^*}(x,y)} \s(\gamma) - \sum_{\gamma \in \Gamma_{\Z^*}^{B(r_0)}(x, y)} \s(\gamma)\\
            &= a(x,y) - \sum_{u \in \partial B(r_0)}H_{B(r_0)^c}(x, u)a(u,y)\\
            &\overset{\ref{lem.a(x,y)}}{=} \frac{a(y)}{\Proba_0[\tau(y) < \tau^+(0)]}\Bigg(\Proba_x[\tau(y) < \tau(0)]
            - \sum_{u \in \partial B(r_0)}H_{B(r_0)^c}(x, u)\Proba_u[\tau(y) < \tau(0)] \Bigg)\\
            &\overset{\ref{lem.hitting.prob}}{=} \fix{a(y)}\Bigg(a(x) + a(y) - a(x-y)
            - \sum_{u \in \partial B(r_0)}H_{B(r_0)^c}(x, u)(a(u) + a(y) - a(y-u)) \Bigg).
    \end{aligned}
  \end{equation}
  By recurrence of the SRW in dimension $d=2$, for all $x \in B(r_0)^c$,
  \be
        \sum_{u \in \partial B(r_0)}H_{B(r_0)^c}(x, u) = \Proba_x\big[\tau(B(r_0))<\infty\big] = 1.
    \ee
    Assume now that $2|x| \leq |y|$. Then $a(x-y) = \frac{2}{\pi}\ln|x-y| + O(1) = \frac{2}{\pi}\ln|y| + O(1)$, $a(y-u) = \frac{2}{\pi} \ln|y| +O(1)$, $a(x) = \frac{2}{\pi}\ln|x| + O(1)$, $a(u) = \frac{2}{\pi}\ln(r_0) + O(1)$, we get
  \begin{equation}
    H_{B(r_0)^c}(x,y) = \frac{2a(y)}{\pi}\Bigg(\ln\frac{|x|}{r_0} + O(1)\Bigg).
  \end{equation}
  If $2|y| \leq |x|$, the proof is the same but $a(x-y) = \frac{2}{\pi}\ln|x-y| + O(1) = \frac{2}{\pi}\ln|x| +O(1)$, so
  \begin{equation}
    H_{B(r_0)^c}(x,y) = \frac{2a(y)}{\pi}\Bigg(\ln\frac{|y|}{r_0} + O(1)\Bigg).
  \end{equation}
\end{proof}

Elaborating on this lemma, we can prove:
\begin{Lem}
  \label{Hr3}
  For all $r_3 \geq 2r_2 \geq 4r_1 \geq 8r_0$, $v \in \partial B(r_1)$, $w \in \partial B(r_2)$,
  \begin{equation}
    H_{B(r_3) \setminus B(r_0)}(v,w) = \frac{4\ln(r_2)}{\pi^2} \Bigg(\ln\frac{r_1}{r_0}\Bigg[1- \frac{\ln\frac{r_2}{r_0}}{\ln\frac{r_3}{r_0}}\Bigg] + \bigg(\ln\frac{r_2}{r_0}\bigg)O\left(\frac{1}{r_0}\right) + O(1)\Bigg).
  \end{equation}
\end{Lem}
\begin{proof}
    \fix{Recall from the proof of Lemma~\ref{Hr0} that $\Gamma_B^C(u,v)$ denote the paths from $u$ to $v$ remaining in $B$ and touching $C$.}
  Let $r_3 \geq 2r_2 \geq 4r_1 \geq 8r_0$, $v \in \partial B(r_1)$, $w \in \partial B(r_2)$. Then
  \begin{equation}
    \label{mainHr3}
    \begin{aligned}
      H_{B(r_3) \setminus B(r_0)}(v,w)
      &= \sum_{\fix{\gamma \in \Gamma_{B(r_3) \setminus B(r_0)}(v, w)}} \s(\gamma)\\
      &= \sum_{\fix{\gamma \in \Gamma_{B(r_0)^c}(v,w)}}\s(\gamma) - \sum_{\fix{\gamma \in \Gamma_{B(r_0)^c}^{B(r_3)^c}(v, w)}}\s(\gamma)\\
      &= H_{B(r_0)^c}(v,w) - \sum_{z \in \partial B(r_3)}H_{B(r_3) \setminus B(r_0)}(v,z)H_{B(r_0)^c}(z,w)\\
      &\overset{\ref{Hr0}}{=} \frac{2a(w)}{\pi}\Bigg[\bigg(\ln\frac{|v|}{r_0} + O(1)\bigg) -  \sum_{z \in \partial B(r_3)}H_{B(r_3) \setminus B(r_0)}(v,z)\bigg(\ln\frac{|w|}{r_0} + O(1)\bigg)\Bigg]\\
      &= \frac{2a(w)}{\pi}\Bigg[\ln\frac{r_1}{r_0} + O(1) - \sum_{z \in \partial B(r_3)}H_{B(r_3) \setminus B(r_0)}(v,z)\bigg(\ln\frac{r_2}{r_0} + O(1)\bigg)\Bigg].
    \end{aligned}
  \end{equation}
  Note that we use Lemma \ref{Hr0} in both directions since $2|v| \leq |w|$ but $2|w| \leq |z|$. By a classical martingale argument (see for example Exercise 1.6.8 of \cite{Lawler}),
    \begin{equation}
        \sum_{z \in \partial B(r_3)}H_{B(r_3) \setminus B(r_0)}(v,z)
        = \Proba_v[\tau(r_3) < \tau(r_0)]
        = \frac{\ln\frac{r_1}{r_0} + O\left(\frac{1}{r_0}\right)}{\ln\frac{r_3}{r_0}} 
        =\frac{\ln\frac{r_1}{r_0}}{\ln\frac{r_3}{r_0}} + O\left(\frac{1}{r_0}\right).
    \end{equation}
    Equation~\eqref{mainHr3} becomes:
    \begin{equation}
        \begin{aligned}
            H_{B(r_3) \setminus B(r_0)}(v,w) 
            &= \frac{2a(w)}{\pi}\Bigg(\ln\left(\frac{r_1}{r_0}\right) - \ln\left(\frac{r_2}{r_0}\right) \Bigg[\frac{\ln\frac{r_1}{r_0}}{\ln\frac{r_3}{r_0}}+ O\left(\frac{1}{r_0}\right)\Bigg] + O(1) \Bigg)\\
            &= \frac{4\ln(r_2) + O(1)}{\pi^2} \Bigg(\ln\left(\frac{r_1}{r_0}\right)\Bigg[1- \frac{\ln\frac{r_2}{r_0}}{\ln\frac{r_3}{r_0}}\Bigg] + \bigg(\ln\frac{r_2}{r_0}\bigg)O\left(\frac{1}{r_0}\right)+ O(1)\Bigg)\\
            &= \frac{4\ln(r_2)}{\pi^2} \Bigg(\ln\left(\frac{r_1}{r_0}\right)\Bigg[1- \frac{\ln\frac{r_2}{r_0}}{\ln\frac{r_3}{r_0}}\Bigg] + \bigg(\ln\frac{r_2}{r_0}\bigg)O\left(\frac{1}{r_0}\right) + O(1)\Bigg)
    \end{aligned}
  \end{equation}
\end{proof}

From here we finally prove the equivalent of Lemma \ref{lemme3.1}:
\begin{Lem}\label{2Dlemme3.1}
    Assume that $\K(x) = o\left(\frac{1}{|x|^2}\right)$. For all $\varepsilon > 0$, there exists $\rho_0$ (depending on $\varepsilon$) such that for all $r_0 \geq \rho_0$, there exist $r_1 < r_2 < r_3$ such that for all $v, v' \in \partial B(r_1)$, for all $w, w' \in \partial B(r_2)$, for all $R$ large enough:
    \begin{equation}
        1- \varepsilon \leq \frac{H^{\K}_{B(r_3) \setminus B(r_0)}(v,w)}{H^{\K}_{B(r_3) \setminus B(r_0)}(v',w')} \leq 1 + \varepsilon.
    \end{equation}
\end{Lem}
\begin{proof}
    \fix{Let} $\delta > 0$. Let $\rho_0$ to be determined in the proof, $r_0 \geq \rho_0$. \fix{Let} $r_1 = \exp\big(\frac{2}{\delta}\big)r_0$, $r_2 = 2r_1$, $r_3 = \frac{1}{2}\exp\Big(\frac{2}{\delta}\Big)r_2$. Then, Lemma \ref{Hr3} gives
    \begin{equation}
        H_{B(r_3) \setminus B(r_0)}(v,w)  = \frac{4\ln(r_2)}{\pi^2}\bigg(\frac{1}{\delta} - \frac{\ln(2)}{2} + \bigg(\ln2 + \frac{2}{\delta}\bigg)O\left(\frac{1}{r_0}\right) + O(1)\bigg).
    \end{equation}
    Recall that $O(1)$ denotes a term bounded by a universal \fix{constant} $C$. For all $\delta$ small enough and $r_0 \geq \rho_0 = \rho_0(\delta)$, this gives for all $v \in \partial B(r_1)$ and $w \in \partial B(r_2)$
    \begin{equation}\label{Hr3fixed}
        1-2\delta \leq \frac{H_{B(r_3) \setminus B(r_0)}(v,w)}{\frac{4\ln(r_2)}{\delta \pi^2}} \leq 1+2\delta.
    \end{equation}
    The key is that this expression does not depend asymptotically on $v,w$ but just on the $r_i$ and their ratios. 
    \fix{The next step is to take the killing into account}.
    \fix{Recall} that
    \begin{equation}
        H_{B(r_3) \setminus B(r_0)}(v,w) = \sum_{\fix{\gamma \in \Gamma_{B(r_3) \setminus B(r_0)}(v, w)}} \s(\gamma).
    \end{equation}
    We will reason as in the proof of Lemma \ref{lemme3.1}. For $\fix{T} > 0$, $v,w \in \fix{\Z^2}$, $B \subset \Z^2$, write
    \begin{equation}
        \fix{
        \Gamma_{B}(v,w,T) = \{\gamma \in \Gamma_B(v,w): \text{ such that } |\gamma| \geq T\}.
        }
    \end{equation}
    We first show that for all $v \in \partial B(r_1)$, $w \in \partial B(r_2)$, we can fix $\fix{T}$ large enough so that
    \begin{equation}\label{eq:A:large}
        \sum_{\fix{\gamma \in \Gamma_{B(r_3) \setminus B(r_0)}(v,w,Tr_0^2)}} \s(\gamma) \leq \delta H_{B(r_3) \setminus B(r_0)}(v,w).
    \end{equation}
    We write
    \begin{equation}
        \sum_{\fix{\gamma \in \Gamma_{B(r_3) \setminus B(r_0)}(v,w,Tr_0^2)}}\s(\gamma)= \sum_{n = \fix{T}r_0^2}^{\infty}\Proba_v[S_n = w, \tau(r_0) \wedge \tau(r_3) \geq n].
    \end{equation}
    Now we split the sum in two parts for some $a=a(\delta) > 0$ to be determined:
    \begin{equation}
        \begin{aligned}
            \sum_{\fix{\gamma \in \Gamma_{B(r_3) \setminus B(r_0)}(v,w,Tr_0^2)}}\s(\gamma) 
            & \leq \sum_{n = \fix{T}r_0^2}^{ar_0^2 \ln(r_0)} \Proba_v[S_n = w] + \sum_{n = ar_0^2 \ln(r_0)}^{\infty}\Proba_v[\tau(r_3) \geq n]\\
            &\overset{\ref{taur3},\ref{LCLT}}{\leq} \sum_{n = \fix{T}r_0^2}^{ar_0^2 \ln(r_0)} \frac{C}{n} + \sum_{n = ar_0^2 \ln(r_0)}^{\infty} C\exp\bigg(-c\frac{n}{r_3^2}\bigg)\\
            &\leq \frac{aC\ln(r_0)}{\fix{T}} + Cr_3^2 \exp\Bigg(-ac\bigg(\frac{r_0}{r_3}\bigg)^2\ln r_0 \Bigg)\\
            &= \frac{aC\ln(r_0)}{\fix{T}} + C\exp\left(\frac{8}{\delta}\right)\frac{r_0^2}{r_0^{ac\exp(-8/\delta)}}
        \end{aligned}
    \end{equation}
  where the penultimate line is obtained by bounding all the terms by the largest one for the first sum and computing the geometric sum. Since $\delta$ is fixed, by~\eqref{Hr3fixed} we know that $H_{B(r_3) \setminus B(r_0)}(v,w) \asymp \ln(r_2)$ (where the constants depend only on $\delta$) so we can choose $a = a(\delta)$ (for example $a = 2/(c\exp(-8/\delta))$), then $\fix{T}$ large enough compared to $a(\delta)$ (for example $\fix{T} = aC/(2\delta)$) and a new $\rho_0 = \rho_0(\delta)$ such that for $r \geq \rho_0(\delta)$, $v \in \partial B(r_1)$, $w \in \partial B(r_2)$,
  \begin{equation}
    \sum_{\fix{\gamma \in \Gamma_{B(r_3)\setminus B(r_0)}(v,w,Tr_0^2)}}\s(\gamma) \leq \delta H_{B(r_3) \setminus B(r_0)}(v,w).
  \end{equation}
  Now, for the paths \fix{$\gamma \in \Gamma_{ B(r_3) \setminus B(r_0)}(v, w) \setminus \Gamma_{ B(r_3) \setminus B(r_0)}(v, w, Tr_0^2)$}, the argument of the proof of Lemma~\ref{lemme3.1} works exactly the same way. Since $\K(x) = o\Big(\frac{1}{|x|^2}\Big)$, there exists $\rho_0$ such that for all $r_0 \geq \rho_0$, $v \in \partial B(r_1)$, $w \in \partial B(r_2)$, for all \fix{$\gamma \in \Gamma_{ B(r_3) \setminus B(r_0)}(v, w) \setminus \Gamma_{ B(r_3) \setminus B(r_0)}(v, w, Tr_0^2)$}, $\B(\gamma) \geq (1-\delta) \s(\gamma)$. Lemma \ref{2Dlemme3.1} follows: for $r_0 \geq \rho_0$ large enough and $r_1, r_2, r_3$ associated, for $v \in \partial B(r_1)$, $w \in \partial B(r_2)$,
    \begin{equation}
        \begin{aligned}
            H_{B(r_3) \setminus B(r_0)}(v,w) 
            &\overset{\eqref{eq:A:large}}{\leq} (1+\delta)\sum_{\fix{\gamma \in \Gamma_{ B(r_3) \setminus B(r_0)}(v, w) \setminus \Gamma_{ B(r_3) \setminus B(r_0)}(v, w, Tr_0^2)}}\s(\gamma)\\
            &\leq \frac{1+\delta}{1-\delta}\sum_{\fix{\gamma \in \Gamma_{ B(r_3) \setminus B(r_0)}(v, w) \setminus \Gamma_{ B(r_3) \setminus B(r_0)}(v, w, Tr_0^2)}}\B(\gamma)
            \leq \frac{1+\delta}{1-\delta} H^{\K}_{B(r_3) \setminus B(r_0)}(v,w).
    \end{aligned}
  \end{equation}
  Besides, $H^{\K}_{B(r_3) \setminus B(r_0)}(v,w) \leq H_{B(r_3) \setminus B(r_0)}(v,w)$ is obvious since $\B(\gamma) \leq \s(\gamma)$ for all path\fix{s}. Lemma~\ref{2Dlemme3.1} follows by using this last equation and~\eqref{Hr3fixed}.
\end{proof}

From this point, the exact same proof as for the case $d \geq 3$ works (the last paragraph of Section~\ref{proof3d}, replacing~\eqref{decomposition} by~\eqref{decomposition2d} and Lemma~\ref{lemme3.1} by Lemma~\ref{2Dlemme3.1}. This concludes the proof of Theorem~\ref{thm.cv} in the case $d=2$.

\bigskip

\paragraph{Declaration.}

\emph{Data availability:} we do not analyse or generate any datasets, because our work proceeds within a theoretical and mathematical approach.

\emph{Conflict of interest:} the authors have no relevant financial or non-financial interests to disclose.

\end{document}